\documentclass[leqno,11pt]{amsart}
\usepackage{mathrsfs}
\usepackage{amsmath}
\usepackage{amssymb,amsthm,upref,graphicx,mathrsfs}
\usepackage{enumerate}
\usepackage{bbm}
\usepackage[mathscr]{eucal}
\usepackage{graphicx}
\usepackage{amsfonts}
\usepackage{bm}
\usepackage{mathabx}
\usepackage{amssymb}
\usepackage{mathrsfs}
\usepackage{color}
\usepackage[utf8]{inputenc}
\usepackage[T1]{fontenc}

\usepackage{CJK}
\usepackage{color}
\usepackage{verbatim}
\usepackage[
  colorlinks=true,
  linkcolor=blue,
  citecolor=blue,
  urlcolor=blue]{hyperref}

\numberwithin{equation}{section}

\newtheorem{theorem}{Theorem}[section]
\newtheorem{proposition}[theorem]{Proposition}
\newtheorem{lemma}[theorem]{Lemma}

\theoremstyle{definition}
\newtheorem{definition}[theorem]{Definition}

\newtheorem{remark}[theorem]{Remark}

\newcommand{\M}{\mathcal{M}}

\newcommand{\A}{\mathcal{A}}

\newcommand{\E}{\mathbb{E}}
\newcommand{\R}{\mathbb{R}}

\newcommand{\Z}{\mathbb{Z}}

\newcommand{\be}{\begin{eqnarray*}}
\newcommand{\ee}{\end{eqnarray*}}
\newcommand{\beq}{\begin{equation}}
\newcommand{\eeq}{\end{equation}}

\newcommand{\N}{\mathcal{N}}

\numberwithin{equation}{section}
\numberwithin{theorem}{section}

\def\F{\mathcal{F}}

\def\Z{\mathbb Z}

\def\H{\mathcal{H}}
\let\<\langle
\def\ch{\raise 0.5ex \hbox{$\chi$}}

\def\E{\mathbb{E}}

\let\\\cr
\let\phi\varphi

\let\epsilon\varepsilon

\def\1{\mathbf{1}}

\makeatletter
\@namedef{subjclassname@2020}{\textup{2020} Mathematics Subject Classification}

\begin{document}

\title[Weighted norm estimates of noncommutative Calder\'{o}n-Zygmund operators]
{Weighted norm estimates of noncommutative Calder\'{o}n-Zygmund operators}

\authors

\author[ W. Fan]{Wenfei Fan}
\address{School of Statistics and Information, Shanghai University of International Business and Economics, Shanghai 201620, China}
\email{fanwenfei@suibe.edu.cn}

\author[ Y. Jiao]{Yong Jiao}
\address{School of Mathematics and Statistics, Central South University, Changsha 410075, China}
\email{jiaoyong@csu.edu.cn}

\author[ L. Wu]{Lian Wu}
\address{School of Mathematics and Statistics, Central South University, Changsha 410075, China}
\email{wulian@csu.edu.cn}

\author[ D. Zhou]{Dejian Zhou}
\address{School of Mathematics and Statistics,
Central South University, Changsha 410075, China}
\email{zhoudejian@csu.edu.cn}



\thanks{This paper is supported by the National Key R$\&$D Program of China (No.2023YFA1010800) and the NSFC (No.12125109, No.12361131578, No.12001541).}

%
\begin{abstract}
This paper is devoted to studying weighted endpoint estimates of operator-valued singular integrals. Our main results include weighted weak-type $(1,1)$ estimate of noncommutative maximal Calder\'{o}n-Zygmund operators, corresponding version of square functions and a weighted $H_1- L_1$ type inequality. All these results are obtained under the condition that the weight belonging to the Muchenhoupt $A_1$ class and certain regularity assumptions imposed on kernels which are weaker than the Lipschitz condition.

\end{abstract}

\maketitle

%
%

\section{Introduction}

The theory of singular integrals, introduced by Calder\'{o}n and Zygmund \cite{CZ1952} in the 1950s, plays a central role in harmonic analysis due to its wide and deep applications in mathematical physics, partial differential equations and other branches of mathematics. The boundedness theory of singular integrals now has been well understood, see for instance the monograph by Grafakos \cite{GL2014} for a detailed exposition of this subject.

The study of weighted estimates related to singular integrals began in the 1970s, soon after Muckenhoupt  \cite{Mu1972} introduced the $A_p$ condition. In \cite{HMW1973}, Hunt, Muckenhoupt and Wheeden proved that Muckenhoupt's $A_p$ condition can be used to characterize the weighted strong-type $(p,p)$ and the weighted weak-type $(p,p)$ estimates of Hilbert transforms. Coifman and Fefferman \cite{CF1874} further extended this result to general Calder\'{o}n-Zygmund operators. The seminar work \cite{HTP2012} due to Hyt\"{o}nen resolves the famous $A_2$ conjecture for general Calder\'{o}n-Zygmund operators. The reader might also consult \cite{GJ1985,ST1989} and references therein for more information on the weighted theory of singular integrals.

On the other hand, noncommutative harmonic analysis has received wide concerns in recent years. A fundamental work in this area is due to Mei \cite{Mei2009}, where the theory of operator-valued Hardy spaces was established. Later, based on the noncommutative probabilistic approaches and the analytic methods originated from the operator space theory, fruitful results related to Calder\'{o}n-Zygmund operators have been transferred to the noncommutative setting (operator-valued context), see for instance \cite{Ca2018,CCP2022,CGPT2023,HLX2020,JMP2014,Pa2009}. However, to the best of our knowledge, the issue about the weighted boundedness of noncommutative Calder\'{o}n-Zygmund operators is almost open. Here, we should mention that very recently, Ga\l{\c{a}}zka et al. \cite{GJOW2022} obtained the sharp weighted Doob inequality for operator-valued martingales.  As applications, they established sharp weighted estimates for the noncommutative Hardy-Littlewood maximal operator and weighted bounds for noncommutative maximal truncations of a certain class of singular integrals.

The purpose of this paper is to investigate the weighted endpoint inequalities of operator-valued Calder\'{o}n-Zygmund operators. To present the results from an appropriate perspective, let us recall some related backgrounds. We refer the reader to the next section for any unexplained terminology. Let $\mathcal{M}=L_\infty(\R^d)\overline\otimes \mathcal{N}$, where $\mathcal{N}$ is a semifinite von Neumann algebra. Let $\Delta$ denote the diagonal of $\R^{d}\times \R^d$. The Calder\'{o}n-Zygmund operator associated with the kernel $K :\R^{d}\times \R^d\setminus \Delta \to \mathbb{C}$ is defined as follows:
\begin{equation*}\label{def-CZO}
Tf(x)=\int_{\R^d}K(x,y)f(y)dy,\quad x\notin \overrightarrow{\mathrm{supp}}(f),
\end{equation*}
whenever $f$ is a $L_1(\M)\cap \M$-valued compactly supported measurable function and $\overrightarrow{\mathrm{supp}}(f)$ denotes the support of $f$ as an operator-valued function in $\R^d$.
For any $\varepsilon>0$, the truncated singular integrals $T_{\varepsilon} f$ associated with the above Calder\'{o}n-Zygmund operator is defined by setting
\begin{equation*}
T_{\varepsilon} f(x)=\int_{|x-y|>\varepsilon} K(x, y) f(y) d y.
\end{equation*}
To study the boundedness of Calder\'{o}n-Zygmund operators $T$ and $T_{\varepsilon}$, one usually needs to employ the following size condition of $K$:
 \begin{equation}\label{sizec}
     	\big|{K}(x,y)\big|\lesssim \frac{1}{|x-y|^d},\quad x,y\in \R^d.
 \end{equation}
Different smoothness conditions are also considered, for example, the Lipschitz regularity condition or the H\"{o}rmander condition:
 \begin{enumerate}[{\rm (i)}]
\item (Lipschitz regularity condition) there exists $0<\gamma\leq 1$ such that
\begin{equation}\label{smoc}
 \begin{split}
 \big|{K}(x,y)-{K}(x',y)\big|&\lesssim \frac{|x-x'|^{\gamma}}{|x-y|^{{d}+\gamma}},\quad |x-x'|\leq \frac{1}{2}|x-y|,\\
 	\big|{K}(x,y)-{K}(x, y')\big|&\lesssim \frac{|y-y'|^{\gamma}}{|x-y|^{{d}+\gamma}},\quad |y-y'|\leq \frac{1}{2}|x-y|.
 \end{split}
 \end{equation}
\item (Classical $L_1$-H\"{o}rmander condition)
\begin{equation}\label{Hor-condition}
	\sup_{y_1,y_2\in \R^d}\int_{|x-y_1|\geq2|y_1-y_2|} \big|{K}(x,y_1)-{K}(x,y_2)\big|dx <\infty.
\end{equation}
\end{enumerate}

In this paper, we always assume that the kernel $K$ satisfies the following regular condition, which is weaker than \eqref{smoc}, and stronger than \eqref{Hor-condition}.

\begin{definition}
Let $1 \leq r < \infty$. We say that the kernel ${K}$ satisfies the $L_r$-H\"{o}rmander condition (write ${K}\in \H_r$), if the following inequality holds
\begin{equation}\label{Lr-Hc}
\sup\limits_{Q \text { dyadic }} \sum\limits_{j \geq 1} \sup \limits_{y \in Q}(2^{j} \ell(Q))^d\left(\frac {1}{2^{j d} \ell(Q)^d} \int_{2^j \ell(Q) \leq\left|x-c_Q\right| \leq 2^{j+1} \ell(Q)}{\big|{K}_Q(x,y)\big|}^r d x\right)^{\frac{1}{r}}<\infty,
 \end{equation}
 where ${K}_Q(x,y)={K}(x, y)-{K}\left(x, c_Q\right)$, and $c_Q$ denotes the centre of $Q$.
\end{definition}
\begin{remark}
Denote by $\mathcal{L}$ the class of kernels satisfying the Lipschitz condition \eqref{smoc} and denote by $\H$ the class of kernels satisfying \eqref{Hor-condition}. It is easy to check that for $1\leq r_1<r_2<\infty$,
\begin{equation*}
\mathcal{L}\subset\H_{r_2}\subset\H_{r_1}\subset\H.
\end{equation*}
\end{remark}
By virtue of Cuculescu projections (\cite{Cuc}), Parcet \cite{Pa2009} formulated the noncommutative Calder\'{o}n-Zygmund decomposition for operator-valued functions and proved that if the kernel $T$ satisfies the Lipschitz condition then the associated
Calder\'{o}n-Zygmund operator is of weak-type $(1,1)$. Since his pioneering work, the study of noncommuative Calder\'{o}n-Zygmund operators has obtained substantial achievements. For instance, Mei and Parcet \cite{Mei2009} established weak-type inequalities for a large class of noncommutative square functions associated with Calder\'{o}n-Zygmund operators. Cadilhac \cite{Ca2018} greatly simplified Parcet's original proof \cite{Pa2009}. Later, with the help of a new noncommutative Calder\'{o}n-Zygmund decomposition, Cadilhac, Conde-Alonso and Parcet \cite{CCP2022} provides a significant improvement on the result obtained in \cite{Pa2009}, by relaxing the {regularity} assumption of kernels from the Lipschitz condition to the $L_2$-H\"{o}rmander condition. We present the result from \cite[Theorem A]{CCP2022} here. Suppose that $K\in \H_2$. Let $T$ be the associated Calder\'{o}n-Zygmund operator which is bounded from $L_2(\mathcal M)$ to $L_2(\mathcal M)$. Then the following inequality holds:
\begin{equation}\label{CCP}
\|Tf\|_{L_{1,\infty}(\M)}\lesssim \|f\|_{L_1(\M)}.
\end{equation}
Recently, Hong, Lai and Xu \cite{HLX2020} obtained the corresponding version of \eqref{CCP} for noncommutative maximal Calder\'{o}n-Zygmund operators. Their result asserts that if the kernel $K\in \H_2$ satisfies the size condition \eqref{sizec}, and there exists some $p_0\in (1,\infty)$ such that $(T_{\varepsilon})_{{\varepsilon}>0}$ is of strong-type $(p_0,p_0)$, then $(T_{\varepsilon})_{{\varepsilon}>0}$ is of weak-type $(1,1)$, that is
\begin{equation}\label{Hong-maximal}
\|(T_\varepsilon f)_{\varepsilon>0}\|_{\Lambda_{1, \infty}\left(\mathcal{M} ; \ell_{\infty}\right)}\lesssim \|f\|_{L_1(\M)}.
\end{equation}
(see Section \ref{sec: 2.3} for the notion of ${\Lambda_{1, \infty}\left(\mathcal{M} ; \ell_{\infty}\right)}$). It is natural to ask whether the estimates \eqref{CCP} and \eqref{Hong-maximal} still hold true when the $L_2$-H\"{o}rmander condition is replaced by the $L_1$-H\"{o}rmander condition (just like their commutative counterparts, see {e.g.} \cite[Chapter V: Corollary 4.11]{GJ1985}). This question remains unknown and was explicitly proposed in \cite[Remark 2.4]{CCP2022}. More results and applications related to noncommutative Calder\'{o}n-Zygmund theory can be found in \cite{Ca2018,CPSZ2019,GAJP2021,HLMP2014,HLM2020,HX2021,JMP2014,JMPX2021} and the references therein.

Our first aim of this paper is to provide weighted analogues of \eqref{CCP} and \eqref{Hong-maximal}. Following \cite{GJOW2022}, we will always consider weights of the form $w\otimes {\bf 1}_{\mathcal{N}}$, where $w$ is a classical weight on $\R^d$ and ${\bf 1}_{\mathcal N}$ stands for the unit of $\mathcal N$. It is obvious that such weights commute with every element of $\mathcal{M}$. However, as evidenced in \cite{GJOW2022}, it is still full of challenges to consider noncommutative weighted estimates in such a special context. In the sequel, with no risk of confusion, we will always identify $w\otimes {\bf 1}_{\mathcal N}$ with $w$. The first main result of this paper reads as follows; in particular, it extends \eqref{Hong-maximal} to the weighted setting.

\begin{theorem}\label{main-maximal}
Let $T$ be a Calder\'{o}n-Zygmund operator associated to a kernel $K$ satisfying \eqref{sizec}. Given a weight ${w}\in A_1$, suppose that $K\in \H_{2r_w'}$, where $r_w$ comes from Lemma \ref{weightprop}(i) and $r_w'$ denotes the conjugate number of $r_w$. If there exists some $p_0>1$ such that $(T_\varepsilon)_{\varepsilon>0}$ is bounded from $L_{p_0}^w(\mathcal M)$ to $L_{p_0}^w(\mathcal M;\ell_\infty)$, then for any $f \in L_1^w(\mathcal{M})$ and $\lambda>0$, there exists a projection $e \in \mathcal{M}$ such that
$$
\sup _{\varepsilon>0}\left\|e\left(T_{\varepsilon} f\right) e\right\|_{\infty} \leq \lambda \quad \text{and} \quad {\lambda}\varphi^w\left(e^{\perp}\right) \lesssim \|f\|_{L_1^w(\M)};
$$
or equivalently, we have
\begin{equation}\label{weighted-main-1}
\|(T_\varepsilon f)_{\varepsilon>0}\|_{\Lambda_{1, \infty}^w\left(\mathcal{M} ; \ell_{\infty}\right)}\lesssim \|f\|_{L_1^w(\M)}.
\end{equation}
\end{theorem}

\begin{remark}
The result above depends on a priori condition which related to the weight $w$, while in the classical case we only need to assume that $(T_\varepsilon)_{\varepsilon>0}$ is of strong-type $(p_0,p_0)$ (see \cite{GJ1985,ST1989}). 
We believe that this is also the case in the noncommutative setting, and we will further explore it in somewhere else. Moreover, according to \cite[Theorem 5.2]{GJOW2022}, a class of maximal singular integrals are bounded from $L_{p_0}^w(\mathcal M)$ to $L_{p_0}^w(\mathcal M;\ell_\infty)$, which illustrates that our priori condition here is meaningful.
\end{remark}

Next, we will switch to the weighted version of \eqref{CCP} and more generally, we are going to work on square functions which are fundamental subjects to the Littlewood-Paley theory.  In 2009, under the Lipschitz regularity assumption, Mei and Parcet \cite{Mei2009} established weak-type inequalities for a large class of noncommutative square functions associated with  Calder\'{o}n-Zygmund operators. This work lays a foundation on the noncommutative Littlewood-Paley theory. Let us briefly recall their result below. Instead of one single kernel $K$, consider a sequence of functions $\vec{K} =(K_k )_{k\geq1}$, where $K_k: \R^{d}\times \R^d\setminus \Delta \to \mathbb{C}$. Define the associated Calder\'{o}n-Zygmund operator by
\begin{equation}\label{def-cz}
\vec{T}f(x)=(T_kf)_{k\geq1},\end{equation}
where
\begin{equation*}
T_kf(x)=\int_{\R^d}K_k(x,y)f(y)dy,\quad x\notin \overrightarrow{\mathrm{supp}}(f).
\end{equation*}
We introduce the following size condition:
\begin{equation*}
 	\|\vec{K}(x,y)\|_{\ell_2}\lesssim \frac{1}{|x-y|^{{d}}},\quad x,y\in \R^d.
\end{equation*}
Similarly, replacing $|\cdot|$ by $\|\cdot\|_{\ell_2}$ on the left side of the inequalities \eqref{smoc} and \eqref{Hor-condition} respectively, one {gets} the corresponding smoothness conditions. It was proved in \cite[Theorem B1]{Mei2009} that 
\begin{equation}\label{Mei-Parcet}
\|\vec{T}f\|_{L_{1,\infty}(\M;\ell_2^{cr})}\lesssim \|f\|_{L_{1}(\M)}
\end{equation}
provided that the associated kernel $K$ satisfies the size condition and the Lipschitz regularity condition (see Section \ref{sec: 2.3} for the definition of ${L_{1,\infty}(\M;\ell_2^{cr})}$).

Our second result of this paper extends \eqref{Mei-Parcet} (and in particular, \eqref{CCP}) to the weighted case. Moreover, unlike the paper \cite{Mei2009} where the Lipschitz condition was employed, our result holds true under the following weaker assumption:
\begin{equation}\label{Lr-Hc-square}
\sup\limits_{Q \,\ \rm { dyadic }} \sum\limits_{j \geq 1} \sup \limits_{y \in Q}(2^{j} \ell(Q))^d\left(\frac {1}{2^{j d} \ell(Q)^d} \int_{2^j \ell(Q) \leq\left|x-c_Q\right| \leq 2^{j+1} \ell(Q)}{\|\vec{K}_{Q}(x,y)\|}_{\ell_2}^r d x\right)^{\frac{1}{r}}<\infty,
 \end{equation}
where $\vec{K}_{Q}(x,y)=\big({K}_k(x, y)-{K}_k\left(x, c_Q\right)\big)_{k\geq 1}$ and $c_Q$ denotes the centre of $Q$.

\begin{theorem}\label{main-square}
Let $T$ be a Calder\'{o}n-Zygmund operator with the kernel $\vec{K} =(K_n)_{n\geq1}$ satisfying the size condition. Given a weight ${w}\in A_1$, if the kernel $\vec{K}$ satisfies \eqref{Lr-Hc-square} with $r=2r_w'$ (where $r_w$ is as given in Lemma \ref{weightprop}(i) and $r_w'$ denotes the conjugate number of $r_w$) and
$\vec T=(T_k)_{k\geq1}$ is bounded from $L_2^w(\M)$ to $L_2^w(\M;\ell_2)$, then
\begin{equation}\label{weighted-main-2}
\|\vec{T}f\|_{L^w_{1,\infty}(\M;\ell_2^{cr})}\lesssim \|f\|_{L^w_{1}(\M)},\quad f\in L^w_{1}(\M).
\end{equation}
\end{theorem}

\begin{remark}
Note that when $K_1=K$ and $K_n=0$ for all $n\geq 2$, the above result can be viewed as a weighted version of \eqref{CCP}. Moreover, if $w\equiv1$, then our results \eqref{weighted-main-1} and \eqref{weighted-main-2} reduce to \eqref{Hong-maximal} and \eqref{Mei-Parcet}, respectively.
\end{remark}

The final aim of this paper is to deal with a weighted $H_1-L_1$ type inequality for noncommutative Calder\'{o}n-Zygmund operators under the H\"{o}rmander condition. Besides \eqref{Lr-Hc-square}, we need the following condition as well:
\begin{equation}\label{Hormander-two}
 \begin{split}
&\sup\limits_{Q \text { dyadic }} \sum\limits_{j \geq 1} \sup \limits_{y \in Q}(2^{j} \ell(Q))^d\left(\frac {1}{2^{j d} \ell(Q)^d} \int_{2^j \ell(Q) \leq\left|x-c_Q\right| \leq 2^{j+1} \ell(Q)}{\|\vec{K}_Q(y,x)\|}_{\ell_2}^r d x\right)^{\frac{1}{r}}<\infty.
 \end{split}
\end{equation}

\begin{theorem}\label{HL1}
Let $T$ be a Calder\'{o}n-Zygmund operator with the kernel $K$ satisfying the size condition. Given a weight ${w}\in A_1$,
if the kernel $K$ satisfies \eqref{Lr-Hc-square} and \eqref{Hormander-two} with $r=r_w'$ (where $r_w'$ is as given in Lemma \ref{weightprop}(i)) and $T:L_2^w(\M)\rightarrow L_2^w(\M;\ell_2)$ is bounded, then
\begin{equation}
\|Tf\|_{L_1^w(\M;\ell_2^{cr})}\lesssim \|f\|_{H^w_{1}(\M)}.
\end{equation}	
\end{theorem}

Even in the unweighted setting, we could not find a concrete proof for the above result. A surprising but reasonable phenomenon of this result is that in the unweighted case the smoothness condition can be relaxed to the $L_1$-H\"{o}rmander condition $\H_1$. Namely, let $T$ be a Calder\'{o}n-Zygmund operator with the kernel $\vec{K} =(K_n )_{n\geq1}$ satisfying the size condition and \eqref{Lr-Hc-square}, \eqref{Hormander-two} with $r=1$, then for $f\in L_{1}(\M)$,
\begin{equation*}\|Tf\|_{L_{1}(\M;\ell_2^{cr})}\lesssim \|f\|_{H_1(\M)}.
\end{equation*}
We think this result is of special interests.

Finally, let us say a few words about the proofs of the main results.
Theorem \ref{main-maximal} and Theorem \ref{main-square} mainly depend on the improved noncommutative Calder\'{o}n-Zygmund decompositions established in \cite{CCP2022}.  The proof of Theorem \ref{HL1} is based on a completely new tool: the weighted atomic decomposition theorem. To establish this tool, we use a constructive approach, which was invented in \cite{CRX2023} and developed to more general case in \cite{RWZ2020}.  The weighted atomic decomposition theorem can be also applied to describe the duality of weighted martingale Hardy spaces and establish new noncommutative martingale inequalities. These applications will be discussed in our future paper.

The paper is organized as follows. In the next section, we present the necessary backgrounds and some basic facts that will be used. In section \ref{sec3}, we recall the noncommutative Calder\'{o}n-Zygmund decomposition established in \cite{CCP2022} and also provide several essential lemmas related to weights for further use. Section \ref{sec4} and Section \ref{sec5} are dedicated to the proofs of Theorem \ref{main-maximal} and Theorem \ref{main-square} separately. In the last section, by means of the new atomic decompositions for weighted martingale Hardy spaces, we prove the third result stated in Theorem \ref{HL1}, and we end this paper with {an open problem}.

\medskip

\section{Preliminaries}\label{sec2}

Throughout this paper, we use $C$ to denote some positive constant which may be different at each occurrence. Without other specific illustration, $C$ will be related with the dimension $d$ or the weight $w$. In addition, we write $A\lesssim B$ if there is a constant $C$ such that the inequality $A \leq C B$ is satisfied, and write $A\approx B$ if both $A\lesssim B$ and $B \lesssim A$ hold.

\subsection{Noncommutative $L_p$ spaces}
In this subsection, we introduce noncommutative $L_p$ spaces. We use standard operator algebra notion and refer to \cite{KRII,Take79} for background on von Neumann algebra theory.
Let $(\mathcal{M}, \tau)$ be a semi-finite von Neumann algebra equipped with a normal faithful trace $\tau$.
All $\tau$-measurable operators are denoted by $L_0(\mathcal{M})$.
Let $1\leq p<\infty$. Define the noncommutative Lebesgue spaces $L_p(\mathcal{M})$ as
$$L_p(\M)=\{a\in L_0(\M): \tau(|a|^p)<\infty\},$$
where $|a|=(a^*a)^{1/2}$.
These are Banach spaces when endowed with the norm given by
$$\|a\|_{L_p(\M)}=[\tau(|a|^p)]^{1/p}.$$
As usual, $L_{\infty}(\mathcal{M})$ is just $\mathcal{M}$ with the usual operator norm. 

Suppose that $a$ is a self-adjoint $\tau$-measurable operator and let $a=\int_{-\infty}^{\infty} \lambda d e_{\lambda}$ stand for its spectral decomposition. For any Borel subset $B$ of $\mathbb{R}$, the spectral projection of $a$ corresponding to the set $B$ is defined by $\chi_B(a)=\int_{-\infty}^{\infty} \chi_B(\lambda) d e_{\lambda}$.
The noncommutative weak $L_p$-space, denoted by $L_{p,\infty}(\M)$, is defined as the collection
of all $a\in L_0(\M)$ for which the quasi-norm
$$\|a\|_{L_{p,\infty}(\M)}=\sup_{\lambda>0}\lambda [\tau(\chi_{(\lambda,\infty)}(|a|))]^{\frac{1}{p}}<\infty.$$

\subsection{Vector-valued noncommutative $L_p$ spaces}\label{sec: 2.3}
We recall the definitions of the column/row spaces. Let $1\leq p<\infty$. Let $(a_n)_{n\geq 1}$ be a finite sequence in $L_p(\M)$. Define
$$
\|a\|_{L_p(\M;\ell_2^c)}=\big\|\big(\sum_{n\geq 1}|a_n|^2\big)^{1/2}\big\|_{L_p(\M)}, \quad
\|a\|_{L_p(\M;\ell_2^r)}=\big\|\big(\sum_{n\geq 1}|a_n^*|^2\big)^{1/2}\big\|_{L_p(\M)}.
$$
Then, it can be checked that $
\|\cdot\|_{L_p(\M;\ell_2^c)}$ and $
\|\cdot\|_{L_p(\M;\ell_2^r)}$ are norms. We define $L_p(\M;\ell_2^c)$ (resp. $L_p(\M;\ell_2^r)$) as the completion of all finite sequences in $L_p(\M)$ with respect to $\|\cdot\|_{L_p(\M;\ell_2^c)}$ (resp. $\|\cdot\|_{L_p(\M;\ell_2^r)}$). In particular, if $p=2$, then by the trace property, it follows that $L_2(\M;\ell_2^c)=L_2(\M;\ell_2^r)$, and we denote it by $L_2(\M;\ell_2)$.

	Next, we define the spaces $L_p(\M;\ell^{cr}_2)$ as follows:
	if $p\geq 2$, define
	$$
	L_p(\M;\ell^{cr}_2)=L_p(\M;\ell_2^c)\cap L_p(\M;\ell_2^r)
	$$
	equipped with the intersection norm
	$$\|a\|_{L_p(\M;\ell^{cr}_2)}=\max\{\|a\|_{L_p(\M;\ell_2^c)}, \|a\|_{L_p(\M;\ell_2^r)}\};$$
	if $1\leq p<2$, define
	$$
	L_p(\M;\ell^{cr}_2)=L_p(\M;\ell_2^c)+ L_p(\M;\ell_2^r)
	$$
	equipped with the sum norm
	$$\|a\|_{L_p(\M;\ell^{cr}_2)}=\inf\{\|b\|_{L_p(\M;\ell_2^c)}+ \|c\|_{L_p(\M;\ell_2^r)}\},$$
	where the infimum is taken over all $b$, $c $ in $L_p(\M)$ such that $a=b+c$.
	
	Replacing $L_p(\M)$ by $L_{p,\infty}(\M)$ in above definitions, we obtain the corresponding weak column/row spaces $L_{p,\infty}(\M;\ell_2^c)$, $L_{p,\infty}(\M;\ell_2^r)$ and $L_{p,\infty}(\M;\ell^{cr}_2)$.
	
The following noncommutative Khintchine inequality established in \cite[Corollary 3.2]{Ca2019} will be used in our later discussions.

	\begin{lemma}\label{Khin}
 Let $\left(\varepsilon_k\right)_{k \geq 1}$ be a sequence of Rademacher functions defined on $(0,1)$, and $a=(a_k)_{k\geq 1}$ be a sequence in $\M$. Then
 $$
\|(a_k)_{k\geq 1}\|_{L_p\left(\mathcal{M};\ell_2^{cr}\right)}\approx\Big\|\sum_{k \geq 1} \varepsilon_k a_k\Big\|_{L_p\left(\mathcal{A}\right)}
$$
holds for all $1\leq p<\infty$, and
$$
\|(a_k)_{k\geq 1}\|_{L_{1, \infty}\left(\mathcal{M}; \ell_2^{c r}\right)} \lesssim\Big\|\sum_{k \geq 1} \varepsilon_k a_k\Big\|_{L_{1, \infty}\left(\mathcal{A}\right)},
$$
where $\mathcal{A}=L_{\infty}(0,1) \overline{\otimes} \mathcal{M}$ equipped with the trace $\widetilde{\varphi}=\int_0^1 \otimes \varphi$.
\end{lemma}

Now, we introduce maximal function spaces in the noncommutative setting. Let $1\leq p\leq \infty$. Following \cite{Ju2002, JX2007}, we define $L_p\left(\mathcal{M}; \ell_\infty\right)$ as the space of all sequences $a = (a_n)_{n \geq 1} \subset L_p\left(\mathcal{M}\right)$ which admit the decomposition: for all $n \geq 1$,
\begin{displaymath}
a_n = ub_n v,
\end{displaymath}
where $u,\,v \in L_{2p}\left(\mathcal{M}\right)$ and $b = \left(b_n\right)_{n \geq 1} \subset L_\infty\left(\mathcal{M}\right)$. We equip this space  with the norm
\begin{displaymath}
\|a\|_{L_p\left(\mathcal{M}; \ell_\infty\right)} = \inf \left\lbrace\|u\|_{L_{2p}(\mathcal{M})} \sup_{n \geq 1} \|b_n\|_{\infty} \|v\|_{L_{2p}(\mathcal{M})}\right\rbrace,
\end{displaymath}
where the infimum runs over all factorizations of $a$ as above.

Furthermore, there is a simpler characterization of $L_p\left(\mathcal{M}; \ell_\infty\right)$ when restricted to self-adjoint operators. Let $a=(a_n)_{n\geq 1}$ be a sequence of self-adjoint operators in $L_p(\M)$. Then $a$ belongs to $L_p\left(\mathcal{M}; \ell_\infty\right)$ if and only if there exists a positive operator $ b\in L_p\left(\mathcal{M}\right)$ such that $-b\leq a_n \leq b$ for all $n \geq 0$. In addition,
\begin{equation}\label{eqiv-space}
\|a\|_{ L_p\left(\mathcal{M}; \ell_\infty\right)} = \inf\lbrace \|b\|_{L_p(\mathcal{M})} : -b\leq a_n \leq b,\quad\forall n\geq 1\rbrace.
\end{equation}
The definition of $L_p(\mathcal{M};\ell_\infty)$ extends easily to the case in which the sequences are indexed by an arbitrary set $I$, the relevant factorization also makes sense. Denote the corresponding space by $L_p(\mathcal{M};\ell_\infty(I))$.  We will omit the index set $I$ if there has no confusion. 

Next, we switch to the weak maximal function space $\Lambda_{p,\infty}(\M;\ell_\infty)$. {Denote by $ \mathcal{P}(\M)$ the set of all projections in $\M$.} Set $e^{\perp}={\bf 1}_\M-e$ for any $e\in \mathcal{P}(\M)$. The space $\Lambda_{p,\infty}(\M;\ell_\infty)$ is defined as the set of all sequences $a=(a_k)$ in $L_{p,\infty}(\M)$ such that the quasi-norm
\begin{equation*}\label{weakspace}
\left\|\left(a_k\right)\right\|_{\Lambda_{p, \infty}\left(\mathcal{M} ; \ell_{\infty}\right)}=\sup _{\lambda>0} \lambda \inf _{e \in \mathcal{P}(\mathcal{M})}\left\{\left(\varphi\left(e^{\perp}\right)\right)^{\frac{1}{p}}:\left\|e a_k e\right\|_{\infty} \leq \lambda, \text{ for all } k\right\}
\end{equation*}
is finite. In particular, we have the following equivalent characterization of $\left\|\left(a_k\right)\right\|_{\Lambda_{p, \infty}\left(\mathcal{M} ; \ell_{\infty}\right)}$ in the case that a sequence $a=(a_k)$ consisting of self-adjoint operators,
\begin{equation*}\label{eqiv-weakspaces}
\left\|\left(a_k\right)\right\|_{\Lambda_{p, \infty}\left(\mathcal{M} ; \ell_{\infty}\right)}=\sup _{\lambda>0} \lambda \inf _{e \in \mathcal{P}(\mathcal{M})}\left\{\left(\varphi\left(e^{\perp}\right)\right)^{\frac{1}{p}}:-\lambda\leq ea_ke\leq \lambda, \text{ for all } k\right\}.
\end{equation*}

\subsection{Weights}\label{sec2.1}
In this subsection, we record some basic information on the classical weighted theory in $\R^d$. A weight is a nonnegative locally integrable function on $\R^d$. Given a weight $w$ and a measurable set $A$ in $\R^d$, we will always use the following notation:
$$w(A)=\int_A w(x){d}x.$$
Given $1<p<\infty$ and a weight $w$ on $\mathbb R^d$, we say that $w$ satisfies Muckenhoupt's condition $A_p$ if its $A_p$ characteristic
\begin{equation*}\label{Ap}
[w]_{A_p} := \sup_Q \left(\frac{1}{|Q|} \int_Q w(x){d}x \right)\left(\frac{1}{|Q|}\int_Q {w(x)}^{\frac{1}{1-p}}dx\right)^{p-1}<\infty,
\end{equation*}
where the supremum is taken over all cubes $Q\subset \R^d$ with sides parallel to the axes and $|Q|$ denotes the measure of $Q$. A weight $w$ belongs to the class $A_1$, if there is a constant $C>0$ such that for all cubes $Q\subset \R^d$ with sides parallel to the axes,
$$ \frac{1}{|Q|}\int_{Q}w(x)dx \leq C\operatorname*{essinf}_{x\in Q}w(x).$$
The smallest $C$ with the above property is denoted by $[w]_{A_1}$ and called the $A_1$ characteristic of $w$. A weight $w$ is said to satisfy the reverse H\"{o}lder inequality for some $q\in(1,\infty)$, denoted by $w\in {\rm{RH}}_q$, if for all cube $Q$ in $\R^d$ we have
$$
\left(\frac{1}{|Q|} \int_Q w(x)^{q} d x\right)^{1 / q} \lesssim \frac{1}{|Q|} \int_Q w(x) d x .$$

\begin{remark}
\begin{enumerate}[{\rm (i)}]
\item It can be checked that ${\rm{RH}}_{q_2}\subset {\rm{RH}}_{q_1}$ for any $1<q_1\leq q_2<\infty$.
\item For $w\in{\rm{RH}}_q$, using the H\"{o}lder inequality, we find
\begin{equation}\label{RH-condition}
\left(\frac{1}{|Q|} \int_Q w(x)^{q} d x\right)^{1 / q} \approx \frac{1}{|Q|} \int_Q w(x) d x .
\end{equation}
\end{enumerate}
\end{remark}

We now state the following useful facts of $A_1$ weights. The proof can be found in \cite[Section 7]{GL2014}.

\begin{lemma}\label{weightprop}
Let $w$ be a weight on $\R^d$ with $w \in A_1$. We have the following assertions.
\begin{enumerate}[{\rm (i)}]
 		\item There exists $ r_w>1$ such that $w\in {\rm{RH}}_{r_w}$. 		
		\item For any cube $Q$ in $\mathbb{R}^d$ and any measurable subsets $S$ of $Q$, we have
$$\frac{|S|}{|Q|} \leq C\frac{w(S)}{w(Q)},$$
where $C$ denotes some positive constant which is independent of $S$, $Q$.
\end{enumerate}
\end{lemma}

\subsection{Noncommutative weighted $L_p$ spaces}\label{sec: 2.4}
Suppose that $\mathcal{N}$ is a given semifinite von Neumann algebra with a faithful normal trace $\nu$. In the rest of the paper, we always assume that $\mathcal{M}=L_{\infty}(\mathbb{R}^d)\overline{\otimes}\mathcal{N}$ endowed with the standard  trace $\varphi=\int_{\R^d}\mbox{d}x \otimes\nu$.

In our considerations below, a weight will be a positive operator of the form $w\otimes {\bf 1}_{\mathcal N}$, where $w$ is a classical weight on $\R^d$ and ${\bf 1}_{\mathcal N}$ stands for the unit of $\mathcal N$.  It is clear that such operators commute with all elements of $\mathcal{M}$. For $1\leq p<\infty$, we say that $ w\otimes {\bf 1}_{\mathcal N}$ satisfies Muckenhoupt's condition $A_p$ (or belongs to the $A_p$ class), if its scalar version $w$ enjoys this property. Furthermore, in this case, we set 
\begin{equation*}
[w\otimes {\bf 1}_{\mathcal N}]_{A_p}=[w]_{A_p}.
\end{equation*}
In the following, with no danger of confusion, we simply write $w$ instead of $w\otimes {\bf 1}_{\mathcal N}$.

Given $1\leq p<\infty$ and $w$ as above, the associated noncommutative weighted $L_p$ space is defined by
$$ L_p^ { {{w}}}(\mathcal{M})=\left\{a\in L_0(\mathcal{M})\,:\,a { {w}}^{1/p}\in L_p(\mathcal{M})\right\},$$
{equipped with the norm
\begin{equation*}
\|a\|_{L_p^w(\M)}=\varphi\big(|aw^{1/p}|^p\big)^{\frac 1 p}=\varphi\big(|a|^pw\big)^{\frac 1 p}.
\end{equation*}
In the following, we denote
\begin{equation*}
\varphi^{w}(a):=\varphi(aw), \quad a\in L_1^w(\M).
\end{equation*}}

By replacing the trace $\varphi$ by its weighted version $\varphi^w$ in section \ref{sec: 2.3}, we get the corresponding weighted vector-valued spaces $L_p^w(\M;\ell^{cr}_2)$, $L_{p,\infty}^w(\M;\ell^{cr}_2)$ and weighted maximal function spaces $L_p^w(\mathcal{M};\ell_\infty)$, $\Lambda_{p,\infty}^w(\mathcal{M};\ell_\infty)$ respectively. Moreover, by Lemma \ref{Khin}, we have the weighted version of noncommutative Khintchine inequality stated below.

\begin{remark}
Let $\left(\varepsilon_k\right)_{k \geq 1}$ be a sequence of Rademacher functions defined on $(0,1)$, and $a=(a_k)_{k\geq 1}$ be a sequence in $\M$. Then
\begin{equation}\label{khin2}
\|(a_k)_{k\geq 1}\|_{L_2^w\left(\mathcal{M};\ell_2\right)}=\Big\|\sum_{k \geq 1} \varepsilon_k a_k\Big\|_{L_2^w\left(\mathcal{A}^w\right)}
\end{equation}
and
\begin{equation}\label{khin1}
\|(a_k)_{k\geq 1}\|_{L_{1, \infty}^w\left(\mathcal{M}; \ell_2^{c r}\right)} \lesssim\Big\|\sum_{k \geq 1} \varepsilon_k a_k\Big\|_{L_{1, \infty}^w\left(\mathcal{A}^w\right)},
\end{equation}
where $\mathcal{A}^w=L_{\infty}(0,1) \overline{\otimes} \mathcal{M}$ equipped with the trace $\widetilde{\varphi^w}=\int_0^1 \otimes \varphi^w$.
\end{remark}

\subsection{Noncommuative (martingale) Hardy spaces}\label{sec2.5}
We recall some basic definitions of noncommutative martingales and Hardy spaces in this subsection. For more details, we refer the reader to \cite{Cuc,Me2007,PX1997}. For each $n\geq 1$, let
	$$D_n=\left\{D_n^{k}: k=(k_1,k_2,\cdots,k_d)\in \mathbb{Z}^d\right\}$$
and $\mathcal{D}_n$ be the $\sigma$-algebra generated by $(D_n^k)_{k}$,
where the dyadic cubes $D_n^{k}$ are defined as
	$$D_n^{k}=2^{-n}\cdot(k_1, k_1+1]\times (k_2, k_2+1]\times\cdots (k_d, k_d+1].$$
Set
\begin{equation}\label{def-filtration}
\M_{n}=L_{\infty}(\R^d, \mathcal{D}_{n},dx)\overline{\otimes} \mathcal{N},\quad n\geq 1
\end{equation}
Let us denote the associated conditional
	expectations as
	$$\mathbb{E}_{n}: \M\to \M_{n}.$$
In the present setting,
		 for {$f \in L_1(\M)+\M$}, we have
\begin{equation}\label{Emn}
	\mathbb{E}_{n}(f)(x)=\sum_{Q\in D_{n}}\frac{1}{|Q|} \int_Q f(y) dy \chi_Q(x), \quad x\in \R^d.
\end{equation}
and for any positive {$f \in L_1(\M)+\M$}
\begin{equation}\label{regular}
\E_n (f)\leq 2^d\E_{n-1}(f),\end{equation}
which implies the regularity of the filtration $(\M_n)_{n\geq 1}$ (see \cite[page 561]{Pa2009}).

A noncommutative martingale with respect to the filtration $(\M_n)_{n\geq 1}$ is a sequence $f = \left(f_n\right)_{n \geq 1}$ in $L_1(\M)+\M$ such that
$$\E_{n}\left(f_{n+1}\right) = f_{n}, \quad \text{for all } n \geq 1.$$
Given a martingale $f = \left(f_n\right)_{n \geq 1}$, the martingale difference sequence $(df_n)_{n\geq1}$ is defined by $df_1=f_1$ and $$df_n=f_n-f_{n-1},\quad \forall n\geq2.$$
In addition, given a weight $w$ and $1\leq p<\infty$, we say that a martingale $f$ is $L_p^w(\M)$-bounded if $\|f\|_{L_p^w(\M)}=\sup_{n\geq 1} \|f_n\|_{L_p^w(\M)}<\infty$.

The column square functions (resp. column conditional square functions) are defined by
	$$
	S_{c,n}(f)=\Big(\sum_{k=1}^n|df_k|^2\Big)^{1/2}, \quad S_{c}(f)=\Big(\sum_{k\geq1}|df_k|^2\Big)^{1/2}.$$
	$$
	\Big(\text{resp.} \quad s_{c,n}(f)=\Big(\sum_{k=1}^n\E_{k-1}|df_k|^2\Big)^{1/2}, \quad s_{c}(f)=\Big(\sum_{k\geq1}\E_{k-1}|df_k|^2\Big)^{1/2}.\quad\Big)
	$$
	The row versions $S_{r,n}(f)$, $S_r(f)$, $s_{r,n}(f)$ and $s_r(f)$ can be defined by taking adjoints.	
Define the weighted martingale Hardy spaces as follows
	$$H_{1w}^c(\M)=\{f\in L_1^w(\M):\|f\|_{H_{1w}^c(\M)}=\|S_c(f)\|_{L_1^w(\M)}<\infty\}$$
	and
		$$H_{1w}^r(\M)=\{f\in L_1^w(\M):\|f\|_{H_{1w}^r(\M)}=\|S_r(f)\|_{L_{1}^w(\M)}<\infty\}.$$
Let
	$$H_{1}^w(\M)=H_{1w}^c(\M)+H_{1w}^r(\M)$$
	equipped with the norm
	$$\|f\|_{H_{1}^w(\mathcal{M})}=\inf\left\{\|g\|_{H_{1w}^c(\mathcal{M})}+\|h\|_{H_{1w}^r(\mathcal{M})}\right\},$$
	where the infimum is taken over all $g\in H_{1w}^c(\M)$ and $h\in H_{1w}^r(\M)$ such that $f=g+h$. Replacing $S_c(f)$ and $S_r(f)$ by $s_c(f)$ and $s_r(f)$ respectively in the above definitions, we obtain the noncommutative martingale spaces $h_{1w}^c(\M)$ and $h_{1w}^r(\M)$ as well.
	
{We assert that the Hardy spaces $H_{1w}^c(\M)$ and $h_{1w}^c(\M) $ are equivalent under some certain condition. To this end, we introduce the martingale $A_1$ condition. A weight $w$ satisfies the martingale $A_1$ condition, if its characteristic
$$\sup_{n\geq 1} \left\|\E_n(w)/w\right\|_{\infty}<\infty.$$
It should be emphasized that if $w$ is a weight on $\R^d$ and $w\in A_1$, then it can be verified that $w$ is a martingale $A_1$ weight with respect to the filtration $(\mathcal M_n)_{n\geq 1}$ (see \cite{MK1977}). More specifically, if $w\in A_1$, then for any $n\geq 1$, there is a constant $C>0$ such that for all cubes $Q\subset \R^d$,
\begin{equation}\label{A1weight}
\sup_{n\geq1}\E_n(w)\leq Cw.
\end{equation}}

\noindent We also need the following lemma, the proof of which is same to that of \cite[Lemma 7.3]{Ju2003}.

\begin{lemma}\label{usefullem}
Let $a$ be an invertible positive operator with bounded inverse. If $b\geq 0$ and $b^2\leq a^2$, then
$$
\varphi^w(a^{-1}(a^2-b^2))\leq 2\varphi^w(a-b).
$$
\end{lemma}

\begin{lemma}\label{use-lem}
Given a weight $w\in A_1$, we have
$$
H_{1w}^c(\M)=h_{1w}^c(\M) $$
with equivalent norms.
\end{lemma}

\begin{proof}
The regularity of the filtration presented in \eqref{regular} gives us
\begin{equation}\label{regularity}
S_c^2(f)\leq 2^ds_c^2(f),
\end{equation}
which implies
$${\|f\|}_{H_{1w}^c(\mathcal{M})}\leq \sqrt{2^d}{\|f\|}_{h_{1w}^c(\mathcal{M})}. $$
Thus, it suffices to prove the converse side. By approximation, we can assume that $S_{c,n}(f)$ and $s_{c,n}(f)$ are invertible with bounded inverses. Then, we rewrite
 \begin{align*}
\|f\|_{h_{1w}^c(\mathcal{M})}&=\varphi(s_c(f)w)=\sum_{n\geq 1}\varphi\big(s_c^{-1}(f)(s^2_{c,n}(f)-s^2_{c,n-1}(f))w\big)\\
&\leq \sum_{n\geq 1}\varphi\big(s_{c,n}^{-1}(f)(s^2_{c,n}(f)-s^2_{c,n-1}(f))w\big)\\
&=\sum_{n\geq 1}\varphi\big(s_{c,n}^{-1}(f)|df_n|^2\E_{n-1}(w)\big).
 \end{align*}
The last equality is due to the fact that the sequence $(s_{c,n}(f))_{n\geq1}$ is predictable. 
Since $w\in A_1$, it follows from \eqref{A1weight} and the fact $
s_{c,n}^{-1}(f)\leq 2^{-\frac d 2}S_{c,n}^{-1}(f)$
that
\begin{align*}
\|f\|_{h_{1w}^c(\mathcal{M})}&\lesssim\sum_{n\geq 1}\varphi\big(S_{c,n}^{-1}(f)|df_n|^2w\big)\\&\lesssim
\sum_{n\geq 1}\varphi\Big(S_{c,n}^{-1}(f)\big(S_{c,n}^2(f)-S^2_{c,n-1}(f)\big)w\Big)\\&\lesssim
\sum_{n\geq 1}\varphi\Big(\big(S_{c,n}(f)-S_{c,n-1}(f)\big)w\Big)= \|f\|_{H_{1w}^c(\mathcal{M})},
\end{align*}
where the last inequality is due to Lemma \ref{usefullem}.
\end{proof}
\medskip

\section{Noncommutative Calder\'{o}n-Zygmund decomposition}\label{sec3}
Our proof of Theorem \ref{main-maximal} and Theorem \ref{main-square} mainly depends on the noncommutative Calder\'{o}n-Zygmund decomposition established in \cite{CCP2022}. To state the decomposition and collect some new properties associated with weights properly, we consider the dyadic filtration $\{\mathcal M_n\}_{n\geq 1}$ defined in \eqref{def-filtration}. This result below can be considered as a weighted version of Cuculescu's seminar work \cite{Cuc}.

 \begin{lemma}[{\cite[Theorem 3.4]{GJOW2022}}]\label{Cuculescu}
 Let $w\in A_1$ and $f\in L_1^w(\mathcal{M})$ be positive. Consider the martingale $(f_n)_{n\geq1}$ with $f_n=\mathbb{E}_n(f)$, $n\geq1$. 	For any fixed $\lambda>0$, there exists a sequence of decreasing projections $(q_n^{(\lambda)})_{n\geq 1}$ in $\mathcal{M}$
 	satisfying the following properties:
 	\begin{enumerate}[{\rm (i)}]
 		\item For every $n\geq1 $, $q_n^{(\lambda)}\in \mathcal{M}_n $;
 		
 		\item For every $n\geq1 $, $q_n^{(\lambda)}$ commutes with $q_{n-1}^{(\lambda)}f_n q_{n-1}^{(\lambda)}$;
 		
 		\item For every $n\geq1 $, $q_n^{(\lambda)}f_nq_n^{(\lambda)}\leq \lambda q_n^{(\lambda)}$;
 		
 		\item If we set $$q^{\lambda}= \bigwedge_{n=1}^{\infty} q_n^{(\lambda)},$$ then
 $$\lambda \varphi^w({\bf 1}_\M -q)\lesssim {[w]}_{A_1}{\|f\|}_{L_1^w(\M)}.$$
 	\end{enumerate}
 \end{lemma}

 In what follows, we simply write $(q_n)_{n\geq 1}$ for the sequence of Cuculescu's projections $(q_n^{(\lambda)})_{n\geq 1}$ and $q$ for the corresponding $q^\lambda$. Set
 \begin{equation}\label{pn}
 	p_n = q_{n-1}-q_n, \quad n\geq 1,
 \end{equation}
 with the convention that $q_0=\bf 1_\M$.
 Then
 \begin{equation}\label{sump}
 	\sum_{n\geq1} p_n={\bf 1_\mathcal M}-q.
 \end{equation}

Let $w\in A_1$ and $f\in L_1^w(\mathcal{M})$ be positive. Consider the following noncommutative Calder\'{o}n-Zygmund decomposition of $f$ as given in \cite{CCP2022}:
\begin{equation}\label{dec}
 		f=g+b_{\rm d}+b_{\rm off}
 	\end{equation}
with
$$g=qfq+\sum_{k\geq1} p_kf_kp_k,$$
$$b_{\rm d}=\sum_{k\geq1}b_{{\rm d},k}=\sum_{k\geq1} p_k(f-f_k)p_k$$
and
$$b_{\rm off}=\sum_{k\geq1}b_{{\rm off},k}+b_{{\rm off},k}^*=\sum_{k\geq1} p_k(f-f_k)q_k+q_k(f-f_k)p_k.$$
Note that by Cuculescue's commutation relation (see Lemma \ref{Cuculescu}(ii)), we also have
$$b_{\rm off}=\sum_{k\geq1}b_{{\rm off},k}+b_{{\rm off},k}^*=\sum_{k\geq1} p_kfq_k+q_kfp_k.$$

Similar to the unweighted setting, using Lemma \ref{Cuculescu}, we may prove the following elementary properties of $g$, $b_{\rm d}$ and $b_{\rm off}$.

 \begin{lemma}\label{NewCZ}
 The decomposition $\eqref{dec}$ satisfies
 \begin{enumerate}[\rm (i)]
 \item $\|g\|_{L_1^w(\mathcal{M})}\lesssim {[w]}_{A_1} \|f\|_{L_1^w(\mathcal{M})}$
		and
$\|g\|_{L_{\infty}(\mathcal{M})}\lesssim \lambda$;

 \item $$\sum_{k\geq 1}\|b_{{\rm d},k}\|_{L^w_1(\mathcal{M})}=\sum_{k\geq1}\|p_k(f-f_k)p_k\|_{L^w_1(\mathcal{M})}\lesssim [w]_{A_1}\|f\|_{L^w_1(\mathcal{M})};
		$$
	\item $\mathbb E_k(b_{{\rm d}, k})=\mathbb E_k(b_{{\rm off}, k})=0$, for each $k\geq 1$.
\end{enumerate}
 \end{lemma}

\begin{proof}
Observe that for each $k\geq 1$, we have
\begin{align*}
\| p_k f_k p_k\|_{L^w_1(\mathcal{M})}&=\varphi^w(p_kf_k)=\varphi(p_kf_kw)=\varphi(p_kf \mathbb E_k(w))\\
& \leq [w]_{A_1}\varphi(p_kf w)=[w]_{A_1} \varphi^w(p_k f).
\end{align*}
The desired estimates of $g$ and $b_{\rm d}$ follow. Other statements can be found in \cite[Lemma 1.1]{CCP2022}.
\end{proof}

 According to \eqref{Emn}, we can then express the projections $q_n$ and $p_n$ as
 \begin{equation}\label{q}
 	q_n=\sum_{Q\in D_n} q_Q\chi_Q
 \end{equation}
 and
 \begin{equation}\label{pe}
 	p_n=\sum_{Q\in D_{n-1}} q_Q\chi_Q- \sum_{Q\in D_{n}} q_Q\chi_Q=\sum_{Q\in D_{n}}( q_{\widehat{Q}}-q_Q)\chi_Q=:\sum_{Q\in D_{n}}p_Q\chi_Q,
 \end{equation}
 where $p_Q\in \mathcal{P}(\N)$ for each $Q\in D_n$ and $\widehat{Q}$ is the unique cube in $D_{n-1}$ such that $Q\subset \widehat{Q}$.
 One can easily verify that for each $n\geq1$ and $Q\in D_{n}$,
 \begin{equation*}\label{pq}
 	p_Qq_Q=q_Qp_Q=0.
 \end{equation*}
 Based on \eqref{q} and \eqref{pe}, we can rewrite $b_{\rm d}$ and $b_{\rm off}$ as follows:
 \begin{eqnarray}\label{bd}
 	b_{\rm d}=\sum_{k\geq1} \sum_{Q\in D_k} p_Q(f-f_Q)p_Q\chi_Q
 \end{eqnarray}
 and
 \begin{eqnarray}\label{boff}
 	b_{\rm off}=\sum_{k\geq1} \sum_{Q\in D_{k}} p_Qfq_Q\chi_Q+q_Qfp_Q\chi_Q.
 \end{eqnarray}
Now set
\begin{equation}\label{def-zeta}
\zeta={\bf 1_\M}-\bigvee_{Q} p_Q\chi_{5Q}.
\end{equation}
Then, we have the following crucial facts.

\begin{lemma}\label{lem1}
The projection $\zeta$ satisfies: 
\begin{enumerate}[\rm (i)]
\item The estimate:
$$
\lambda \varphi^w\left(\mathbf{1}_{\mathcal{M}}-\zeta\right) \lesssim [w]_{A_1}\|f\|_{L_1^w(\mathcal{M})};
$$

\item The cancellation property:
$$\zeta(x) p_n(y)=p_n(y) \zeta(x)=0$$
whenever $y \in Q \in {D}_n$ and $x \in 5 Q$;

\item The cancellation properties:
$$\zeta(x) b_{\emph{d},n}(y)\zeta(x)=0,\quad \zeta(x) (b_{\emph{off},n}(y)+b_{{\rm off},n}^*(y))\zeta(x)=0$$
whenever $x,y\in \R^d$ such that $y\in 5\sqrt{d} Q_{x,n}$. Here and below, $Q_{x,n}\in D_n$ stands for the cube containing $x$.
\end{enumerate}
 \end{lemma}

\begin{proof}
The proof is similar to the unweighted case (\cite[Lemma 3.4, Lemma 3.8]{Ca2018} or \cite[Lemma 2.1]{CCP2022}). We include a short proof of (i) for the convenience of the reader. By Lemma \ref{weightprop}(ii), we have
\begin{align*}
\lambda \varphi^w\left(\mathbf{1}_{\mathcal{M}}-\zeta\right)&=
\lambda \varphi^w\Big(\sum_{Q} p_Q\chi_{5Q}\Big)
=\lambda \sum_Q\nu(p_Q)w(5Q)\\
&\lesssim \lambda \sum_Q\nu(p_Q)w(Q)=\lambda \varphi^w({\bf 1_\M}-q)\lesssim [w]_{A_1}\|f\|_{L_1^w(\M)},
\end{align*}
where the last inequality is due to Lemma \ref{Cuculescu}(iv).
\end{proof}

\medskip

\section{Weighted estimate for maximal singular integral operators}\label{sec4}

In this section, we provide the proof of the weighted weak-type estimate for maximal singular integrals, namely Theorem \ref{main-maximal}.  Before that, we state several observations first.

Let $\psi$ be a smooth radial positive function supported in the annulus $[1, 2]$ such that
$$
\sum_{i\in\Z}\psi_i(x):=\sum_{i\in\Z}\psi\big(\frac{2^i x}{\sqrt{d}}\big)=1, \quad \forall x\neq 0.
$$
Such $\psi$ will be fixed throughout this section. For any $i\in\Z$, set
$$
K_i(x,y)=K(x,y)\psi_i(x-y).
$$
Then, given a reasonable $f\in L_1^w(\M)$, we rewrite
$$
T_\varepsilon f(x)=\sum_{i\in\Z} \int_{|x-y|>\varepsilon}K_i(x,y)f(y)dy.$$
Set $j_\varepsilon=[\rm{log}_2(\frac {2\sqrt{d}} {\varepsilon})]$. Then, it is easy to see that, for any fixed $x\in \R^d\setminus\{0\}$ and $i\geq j_\varepsilon+1$,
$$
 \{y\in \R^d: K_i(x,y)\neq 0\} \cap \{y\in\R^d: |x-y|>\varepsilon\}= \emptyset;
$$
and for $i\leq j_\varepsilon-1$, we have
$$
 \{y\in \R^d: K_i(x,y)\neq 0\} \subset \{y\in\R^d: |x-y|>\varepsilon\}.
$$
Therefore, we may write
\begin{equation}\label{reduction}
T_\varepsilon f(x)=\sum_{i< j_\varepsilon} \int_{\R^d}K_i(x,y)f(y)dy+\int_{|x-y|>\varepsilon}K_{j_\varepsilon}(x,y)f(y)dy.
\end{equation}
For brevity, we denote the second term of \eqref{reduction} as $T_{\varepsilon,j\varepsilon}f(x)$. The following result enables us to reduce the study of $(T_\varepsilon)_{\varepsilon>0}$ to its lacunary subsequence.

\begin{proposition}\label{keyprop}
Given a weight ${w}\in A_1$, let $T$ be a Calder\'{o}n-Zygmund operator satisfying the assumptions in Theorem \ref{main-maximal}. Then for any operator-valued positive function $f \in L_1^w(\mathcal{M})$, we have the following estimates:
\begin{enumerate}[{\rm (i)}]
\item  For any $\lambda>0$, there exists a projection $e\in \mathcal{M}$ such that
$$
\sup _{\varepsilon>0}\left\|e\left(T_{\varepsilon,j\varepsilon}f\right) e\right\|_{\infty} \lesssim \lambda \text { and } \lambda\varphi^w\left( {\bf 1_\M}-e\right) \lesssim {\|f\|_{L_1^w(\M)}}.
$$
\item $\left(T_{\varepsilon,j\varepsilon}\right)_{\varepsilon>0}$ is bounded from $L_{p}^w(\mathcal M)$ to $L_{p}^w(\mathcal M;\ell_\infty)$ for all $1<p < \infty$.
\end{enumerate}
\end{proposition}

\begin{proof}
The discussion in \cite[Proposition 3.2]{HLX2020} shows that for any $\varepsilon>0$,
$$
-M_{2^{-j_\varepsilon+1}\sqrt{d}}\left(T_{\varepsilon,j\varepsilon}\right)_{\varepsilon>0}f(x)\lesssim
T_{\varepsilon,j\varepsilon}f(x)\lesssim M_{2^{-j_\varepsilon+1}\sqrt{d}}\left(T_{\varepsilon,j\varepsilon}\right)_{\varepsilon>0}f(x),
$$
where $M_r$ is the Hardy-Littlewood averaging operator defined as
$$
M_{r} f(x)=\frac{1}{r^d} \int_{|x-y| \leq r} f(y) d y, \quad x\in \mathbb{R}^d.
$$
Then, by the noncommutative weighted Hardy-Littlewood maximal inequalities established in \cite[Theorem 4.1]{GJOW2022}, we obtain the desired results.
\end{proof}

\begin{lemma}\label{keylemma2}
Let $1\leq p<\infty$. Given a weight ${w}\in A_1$, let $f\in L_1^w(\M)$ be positive. Suppose that $K$ satisfies the size condition and \eqref{Lr-Hc} with $r=p r_w'$, where $r_w$ is as given in Lemma \ref{weightprop} and $r_w'$ is the conjugate of $r_w$. For any $Q\in D_n$, denote by $\{Q_i\}_{i\leq n-1}$ a sequence of sets satisfying $Q\subset Q_i$, $|Q_i|\approx 2^{-id}$. Then
 $$\sum\limits_{i\leq n-1}|Q_i|\left(\frac{1}{|Q_i|}\int_{Q}{\|f(y)\|}_{L_1(\N)}\int_{Q_i}\big|K_{i, Q}(x, y)\big|^p w(x)d xdy\right)^{\frac 1 p}\lesssim {\|\chi_Qf\|}_{L_1^w(\M)}^{\frac 1 p},$$
 where $K_{i, Q}(x, y)=K_i(x, y)-K_i\left(x, c_{Q}\right)$, and $c_{Q}$ denotes the centre of $Q$.
\end{lemma}
\begin{proof}
For the sake of brevity, denote
$$
{\rm (LHS)}:=\sum\limits_{i\leq n-1}|Q_i|\left(\frac{1}{|Q_i|}\int_{Q}{\|f(y)\|}_{L_1(\N)}\int_{Q_i}\big|K_{i, Q}(x, y)\big|^p w(x)d xdy\right)^{\frac 1 p}.
$$
Applying the H\"{o}lder inequality and Lemma \ref{weightprop}(i), we get
\begin{align*}
&\int_{Q_i}\left|K_{i, Q}(x, y)\right|^{p} w(x)d x\\
&\leq \big|{Q_i}\big|\left(\frac{1}{\big|{Q_i}\big|}\int_{\R^d}\left|K_{i, Q}(x, y)\right|^{pr_w'}dx\right)^{\frac{1}{r_w'}}
\left( \frac{1}{\big|{Q_i}\big|}\int_{Q_i}w(x)^{r_w} dx \right)^{r_w}\\
&\lesssim \big|{Q_i}\big|\left(\frac{1}{\big|{Q_i}\big|}\int_{\R^d}\left|K_{i, Q}(x, y)\right|^{pr_w'}dx\right)^{\frac{1}{r_w'}}
\left( \frac{1}{\big|{Q_i}\big|}\int_{Q_i}w(x) dx \right).
\end{align*}
As $w\in A_1$, for any $y\in Q\subset Q_i$, we have
$$
 \left( \frac{1}{\big|{Q_i}\big|}\int_{Q_i}w(x) dx \right)
 \leq Mw(y)\lesssim w(y),
$$
where $M$ is the Hardy-Littlewood maximal operator. Thus,
\begin{align*}
&\int_{Q_i}\left|K_{i, Q}(x, y)\right|^{p} w(x)d x\lesssim \big|{Q_i}\big|\left(\frac{1}{\big|{Q_i}\big|}\int_{\R^d}\left|K_{i, Q}(x, y)\right|^{pr_w'}dx\right)^{\frac{1}{r_w'}}
w(y),
\end{align*}
which gives
\begin{align*}
{\rm (LHS)}\lesssim \sum_{i\leq n-1} \big|{Q_i}\big|\left(\int_{Q}{\|f(y)\|}_{L_1(\N)}\left(\frac{1}{\big|{Q_i}\big|}\int_{\R^d}\left|K_{i, Q}(x, y)\right|^{pr_w'}dx\right)^{\frac{1}{r_w'}} w(y)dy\right)^{\frac 1p}. \end{align*}
Thanks to \cite[Lemma 4.2]{HLX2020}, for any integer $i\leq n-1$, we have
\begin{equation*}\label{keyclaim}
\int_{\R^d}\left|K_{i, Q}(x, y)\right|^{r}dx\lesssim
2^{id(r-1)}[{\rm{m}}_{r}(n-i)+2^{(i-n)}]^{r},
\end{equation*}
where for $r\geq 1$,
$$
{\rm{m}}_{r}^i(j)=\sup \limits_{y \in Q}\left({\big(2^{j} \ell(Q)\big)}^{d(r-1)} \int_{2^j \ell(Q) \leq\left|x-c_Q\right| <2^{j+1} \ell(Q)}{|{K}_{i,Q}(x,y)|}^r d x\right)^{\frac{1}{r}}.
$$
Then, taking $r=pr_w'$, since $|{Q_i}|\approx 2^{-id}$, we further get
$$
{\rm (LHS)}\lesssim \sum\limits_{i\leq n-1} [{\rm{m}}_{pr_w'}^i(n-i)+2^{(i-n)}]
\left(\int_Q{\|f(y)\|}_{L_1(\N)}w(y)dy\right)^{\frac 1 p}.
$$
Since $K$ satisfies the $L_{p r_w'}$-H\"{o}rmander condition, it follows that
\begin{equation}\label{Lr}
\sum_{j\geq 1}{{\rm{m}}_{p r_w'}^i(j)}<\infty.
\end{equation}
Changing variables and using \eqref{Lr}, the desired inequality follows:
$$
{\rm (LHS)}\lesssim\sum_{k\geq 1}[{\rm{m}}_{pr_w'}^i(k)+2^{-k}]{\|\chi_Qf\|}_{L_1^w(\M)}^{\frac 1 p}\lesssim {\|\chi_Qf\|}_{L_1^w(\M)}^{\frac 1 p}.
$$
The assertion is verified.
\end{proof}

Now, we are ready to prove our first main result: Theorem \ref{main-maximal}.

\begin{proof}[Proof of Theorem \ref{main-maximal}]
For any $j\in \Z$, we introduce the following operator $T_{j}$:
$$T_{j}f(x)=\sum_{i< j} \int_{\R^d}K_i(x,y)f(y)dy.$$
Without loss of generality, we may assume that the kernel $K$ is real and $f$ is positive (see \cite{HLX2020}). According to Proposition \ref{keyprop}(i) and the quasi-triangle inequality in $\Lambda_{1,\infty}^w(\mathcal M;\ell_\infty)$, it is sufficient to show that there exists a projection $e\in \M$ such that for any $f\in L_1^w(\M)$
\begin{equation}\label{aimeq1}
\sup_{j\in\mathbb Z}\left\|e\left(T_j f\right) e\right\|_{\infty}\leq \lambda\quad \text{and} \quad \lambda \varphi^w\left({\bf1}_{\M}-e\right) \lesssim \|f\|_{L_1^w(\mathcal M)}.
\end{equation}

Consider the noncommutative Calder\'{o}n-Zygmund decomposition \eqref{dec} of $f$:
$$f=g+b_{\mathrm{d}}+b_{\mathrm{off}}.$$
To prove \eqref{aimeq1}, it suffices to show that for any $h\in \{g, b_{\mathrm{d}}, b_{\mathrm{off}}\}$, there exists a projection ${e}_h \in \mathcal{M}$
such that
\begin{equation}\label{aimeq}
\sup_{j\in\Z}\left\|{e}_h \left(T_j h\right) {e}_h\right\|_{\infty}\leq \lambda\quad \text{and} \quad \lambda \varphi^w\left(\textbf{1}_{\M}-e_h\right) \lesssim \|f\|_{L_1^w(\mathcal M)}.
\end{equation}
Indeed, once we have \eqref{aimeq}, then the desired result \eqref{aimeq1} follows by setting $e=e_g \wedge e_{b_{\rm d}} \wedge e_{b_{\rm off}}$.
In the following, we establish \eqref{aimeq} for $h=g, b_{\mathrm{d}}$ and $b_{\mathrm{off}}$ separately.

\medskip
{\bf Estimate of \eqref{aimeq}: the case of \bm{$h=g$}.} Since $(T_\varepsilon)_{\varepsilon>0}$ is bounded from $L_{p_0}^w(\M)$ to $L_{p_0}^w(\M;\ell_\infty)$ for some $p_0>1$, by the triangle inequality in $L_p(\M;\ell_\infty)$ and Proposition \ref{keyprop}(ii), we conclude that $(T_j)_{j\in\Z}$ is also bounded from $L_{p_0}^w(\M)$ to $L_{p_0}^w(\M;\ell_\infty)$. Therefore, there exists a positive operator $a \in L_p^w\left(\mathcal{M}\right)$ such that for any $j\in\Z$,
$$-a\leq T_j g\leq a\quad \text{and}\quad  {\|a\|}_{L_{p_0}^w(\M)}\lesssim{\|g\|}_{L_{p_0}^w(\M)}.$$
Take $e_g=\chi_{(0,\lambda]}(a)$. Then it follows that $\left\|e_g (T_j g) e_g\right\|_{\infty} \leq \lambda$. Moreover, as a consequence of the Chebychev inequality, we have
$$
\lambda\varphi^w\left(\textbf{1}_\M-e_1\right) \leq \frac{{\|a\|}_{L_{p_0}^w(\M)}^{p_0}}{\lambda^{{p_0}-1}}\lesssim \frac{{\|g\|}_{L_{p_0}^w(\M)}^{p_0}}{\lambda^{{p_0}-1}}\leq \frac{{\|g\|}_{L_1^w(\M)}{\|g\|}_{\infty}^{{p_0}-1}}{\lambda^{{p_0}-1}}\lesssim [w]_{A_1}\|f\|_{L_1^w(\M)},
$$
where the last inequality is due to Lemma \ref{NewCZ}(i). This completes the proof of \eqref{aimeq} for $h=g$.

\medskip
{\bf Estimate of \eqref{aimeq}: the case of \bm{$h=b_d$}.} We rewrite
\begin{equation}\label{bd-deco}
{T_j}(b_{\mathrm{d}})=T_j(b_{\mathrm{d}})(\textbf{1}_{\M}-\zeta)+(\textbf{1}_{\M}-\zeta)T_j(b_{\mathrm{d}})\zeta+\zeta T_j(b_{\mathrm{d}})\zeta,
\end{equation}
where $\zeta$ is defined in \eqref{def-zeta}. Recall that $K$ is real and
$$b_{\mathrm{d}}=\sum_{n\geq1}b_{\mathrm{d},n} =\sum_{n\geq1}\sum_{Q\in D_n}b_{\mathrm{d},n,Q}.$$
It is easy to see that $\zeta T_j(b_{\mathrm{d}})\zeta$ is self-adjoint. Denote by
$$
G:=\sum_{n\geq 1}\sum_{i\leq n-1}\sum_{Q\in D_n}\left|\int_{Q}K_{i,Q}(x,y)b_{\mathrm{d},n,Q}(y) d y\right|,
$$
where $K_{i,Q}$ is given in Lemma \ref{keylemma2}. The following estimate is contained in \cite[Section 4.2]{HLX2020}
\begin{equation}\label{claim-1}
-\zeta G \zeta \leq\zeta T_jb_{\mathrm{d}}\zeta \leq \zeta G \zeta.
\end{equation}
We claim that
\begin{equation}\label{claim-2}
{\|G\|}_{L_1^w(\M)}\lesssim {\|f\|}_{L_1^w(\M)}.
\end{equation}

Before verifying the claim, let us finish the estimate for $b_d$. Set $$e_{b_{\mathrm{d}}}=\chi_{(0,\lambda]}(\zeta G\zeta)\wedge\zeta.$$
Then, by \eqref{bd-deco} and \eqref{claim-1}, we have
$$\sup_{j\in\mathbb{Z}}\left\|e_{b_{\mathrm{d}}} \left(T_j b_{\mathrm{d}}\right) e_{b_{\mathrm{d}}} \right\|_{\infty}=\sup_{j\in\mathbb{Z}} \left\|e_{b_{\mathrm{d}}} \zeta\left(T_j b_{\mathrm{d}}\right)\zeta e_{b_{\mathrm{d}}} \right\|_{\infty} \leq \lambda.$$
On the other hand, from Lemma \ref{Cuculescu}(iv), the Chebychev inequality and \eqref{claim-2}, it follows that
\begin{align*}
 \lambda \varphi^w(\textbf{1}_{\M}-e_{b_{\mathrm{d}}})& \leq \lambda \varphi^w(\textbf{1}_{\M}-\zeta)+\lambda \varphi^w(\chi_{(\lambda,\infty)}(\zeta G\zeta))\\
 & \lesssim [w]_{A_1}  {\|f\|}_{L_1^w(\M)}+ {\|\zeta G\zeta\|}_{L_1^w(\M)}\\
 & \leq [w]_{A_1}  {\|f\|}_{L_1^w(\M)}+ {\| G\|}_{L_1^w(\M)}\lesssim {\|f\|}_{L_1^w(\M)}.
\end{align*}
This proves \eqref{aimeq} for $h=b_{\mathrm{d}}$.

It remains to verify the claim \eqref{claim-2}. For the sake of convenience, for any fixed $n\geq 1$, $i\leq n-1$ and $Q\in D_n$, denote by
 $$
 G_Q(x):=\int_{Q}\left|K_{i,Q}(x,y)\right| b_{\mathrm{d},n,Q}(y)d y.
 $$
 We observe that there exists $Q_i$ satisfying $ Q_i\supset Q$ and $ |Q_i|\approx 2^{-id}$ such that
 \begin{equation}\label{GQ}
 G_Q=G_Q \chi_{Q_i}.
\end{equation}
 In fact, fixing $x\in \R^d$ satisfying $G_Q(x)\neq 0$, there exists $y_0\in Q$ such that $K_{i,n}(x,y_0)\neq 0$. This implies
 $$
 2^{-i}\sqrt{d}\leq |x-y_0|\leq 2^{-i+1}\sqrt{d} \quad \mbox{or} \quad  2^{-i}\sqrt{d}\leq |x-c_{Q}|\leq 2^{-i+1}\sqrt{d}.
 $$
Suppose the first case holds; namely,
 $$
 2^{-i}\sqrt{d}\leq |x-y_0|\leq 2^{-i+1}\sqrt{d}. $$
On one hand,
 $$
 \mathrm{dist}(x,Q)\leq  |x-y_0|\leq 2^{-i+1}\sqrt{d}.
 $$
On the other, it is obvious that
 $$|x-y_0|\leq \mathrm{dist}(x,Q)+ \sqrt{d}\ell(Q).$$
Then, since $i\leq n-1$, it follows that
$$\mathrm{dist}(x,Q)\geq |x-y_0| -\sqrt{d}\ell(Q)\geq 2^{-i}\sqrt{d}- 2^{-i-1}\sqrt{d}=2^{-i-1}\sqrt{d}.$$
Thus, we have $$\mathrm{dist}(x,Q)\approx 2^{-i}$$
with the constants depends only on $d$. The case $2^{-i}\sqrt{d}\leq |x-c_{Q}|\leq 2^{-i+1}\sqrt{d}$ can be proved similarly. Thus, the support set of $G_Q$ is contained in the set $\{x\in \R^d: \mathrm{dist}(x,Q)\approx 2^{-i}\}$. Set
\begin{equation*}\label{Qi}
{Q_i}=Q\cup\{x:{\rm dist}(x,Q)\approx 2^{-i}\}.
\end{equation*}
The observation has been verified.  Thanks to this observation, we conclude that
 \begin{align*}
{\|G\|}_{L_1^w(\M)}&\leq  \sum_{n\geq 1}\sum_{Q\in D_n}\sum_{i\leq n-1}\nu\int_{\R^d} G_Q(x)w(x)dx \\
&=   \sum_{n\geq 1}\sum_{Q\in D_n}\sum_{i\leq n-1}\nu\int_{Q_i} G_Q(x)w(x)dx\\
&= \sum_{n\geq 1}\sum_{Q\in D_n}\sum_{i\leq n-1}\int_{Q_i}\int_{Q}\left|K_{i,Q}(x,y)\right|{\|b_{\mathrm{d},n,Q}(y)\|}_{L_1(\N)} d yw(x)dx\\
&=\sum_{n\geq 1}\sum_{Q\in D_n}\sum_{i\leq n-1}
\int_{Q_i}{\|b_{\mathrm{d},n,Q}(y)\|}_{L_1(\N)}\int_{Q}\left|K_{i,Q}(x,y)\right| w(x)dxdy.
\end{align*}
Consequently, using Lemma \ref{keylemma2} for $p=1$, we get
 \begin{align*}
{\|G\|}_{L_1^w(\M)}&\lesssim\sum_{n\geq 1}\sum_{Q\in D_n}{\|b_{\mathrm{d},n,Q}(y)\|}_{L_1^w(\M)}\\
&\leq \sum_{n\geq 1}{\|b_{\mathrm{d},n}(y)\|}_{L_1^w(\M)} \lesssim {\|f\|}_{L_1^w(\M)},
\end{align*}
where the last inequality is due to Lemma \ref{NewCZ}(ii). This verifies \eqref{claim-2}; and so the proof of the estimate for the diagonal part of the bad function is complete.

\medskip

{\bf Estimate of \eqref{aimeq}: the case of \bm{$h=b_\mathrm{off}$}.}  Set $b_{\mathrm{off},n,Q}=p_Qfq_Q$. Then
$$b_\mathrm{off}=\sum_{n\geq1}\sum_{Q\in D_n} b_{\mathrm{off},n,Q}+b_{\mathrm{off},n,Q}^*.$$
Denote by
$$
H:=\sum_{n\geq 1}\sum_{i\leq n-1}\sum_{Q\in D_n}\left|\int_{Q}K_{i,Q}(x,y)\big(b_{\mathrm{off},n,Q}(y)+b_{\mathrm{off},n,Q}^*(y)\big)  d y\right|.
$$
Similar as above, we have
\begin{equation}\label{claim-3}
-\zeta H \zeta \leq\zeta T_jb_{\mathrm{off}}\zeta \leq \zeta H\zeta, 
\end{equation}
and we claim that
\begin{equation}\label{claim-4}
{\|H\|}_{L_1^w(\M)}\lesssim {\|f\|}_{L_1^w(\M)}.
\end{equation}
Once the claim is verified, taking $$e_{b_{\mathrm{off}}}=\chi_{(0,\lambda]}(\zeta H\zeta)\wedge\zeta,$$
we get \eqref{aimeq} for $h=b_{\mathrm{off}}$. 
Therefore, it remains to verify the claim. Note that
\begin{equation*}
  {\|H\|}_{L_1^w(\M)}\lesssim \sum_{n\geq 1}\sum_{i\leq n-1}\sum_{Q\in D_n}
  \int_{\R^d}{\left\|\int_{Q}K_{i,Q}(x,y)p_Qf(y)q_Q d y\right\|}_{L_1(\N)}w(x)dx.
   \end{equation*}
Using the H\"{o}lder inequality from \cite[Proposition 1.1]{Me2007}, we deduce that
\begin{align*}
&{\left\|\int_{Q}K_{i,Q}(x,y)p_Qf(y)q_Q d y\right\|}_{L_1(\N)}\\
&\leq
	\Big\|\Big(\int_Q |K_{i,Q}(x,y)|^2p_Qf(y)p_Q dy\Big)^{\frac 1 2}\Big\|_{L_1(\N)} \cdot	\Big\|\Big(\int_Q q_Qf(y)q_Q dy\Big)^{\frac 1 2}\Big\|_{L_{\infty}(\N)}\\
	&\leq \Big\|\Big(\int_Q |K_{i,Q}(x,y)|^2p_Qf(y)p_Q dy\Big)^{\frac 1 2}\Big\|_{L_2(\N)}[\nu(p_Q)]^{\frac12} (\lambda|Q|)^{\frac12}.
\end{align*}
Therefore,
\begin{equation*}
  {\|H\|}_{L_1^w(\M)}\lesssim \sum_{n\geq 1}\sum_{i\leq n-1}\sum_{Q\in D_n}(\nu(p_Q)\lambda|Q|)^{\frac12}\int_{\R^d}\Big\|\Big(\int_Q |K_{i,Q}(x,y)|^2p_Qf(y)p_Q dy\Big)^{\frac 1 2}\Big\|_{L_2(\N)}w(x)dx.
   \end{equation*}
For any fixed $n\geq 1$, $i\leq n-1$ and $Q\in D_n$, denote by
 $$
 H_Q(x):=\int_{Q}\left|K_{i,Q}(x,y)\right|^2 p_Q f(y) p_Q d y.
 $$
Similar to \eqref{GQ}, we can show that there exists $Q_i$ satisfying $ Q_i\supset Q$, $ |Q_i|\approx 2^{-id}$ such that
 $$
 H_Q=H_Q \chi_{Q_i}.
 $$
Based on this observation, one can proceed with the estimate of $H$ as follows:
\begin{align}\label{key-h}
  &{\|H\|}_{L_1^w(\M)}\nonumber\\
  &\lesssim \sum_{n\geq 1}\sum_{i\leq n-1}\sum_{Q\in D_n}(\nu(p_Q)\lambda|Q|)^{\frac12}\int_{Q_i}\Big\|\Big(\int_Q |K_{i,Q}(x,y)|^2p_Qf(y)p_Q dy\Big)^{\frac 1 2}\Big\|_{L_2(\N)}w(x)dx\\
  &= \sum_{n\geq 1}\sum_{Q\in D_n}\sum_{i\leq n-1} (\nu(p_Q)\lambda|Q|)^{\frac12}\int_{Q_i}\Big\|\Big(\int_Q |K_{i,Q}(x,y)|^2p_Qf(y)p_Q dy\Big)^{\frac 1 2}\Big\|_{L_2(\N)}w(x)dx\nonumber.
   \end{align}
According to Lemma \ref{weightprop}(ii), for any $ Q_i\supset Q$, we have
\begin{equation}\label{hhh}
|Q|^{1/2}\leq C\frac{w(Q)^{1/2}|Q_i|^{1/2}}{w(Q_i)^{1/2}},
\end{equation}
where $C$ is a positive constant which is independent of $i$. Thus, combining \eqref{hhh} with the Cauchy-Schwarz inequality, it follows that
\begin{align*}
&|Q|^{1/2}\sum_{i\leq n-1}\int_{Q_i}{\left\|\Big(\int_{Q}|K_{i,Q}(x,y)|^2 p_Qf(y)p_Q d y\Big)^{\frac12}\right\|}_{L_2(\N)}w(x)dx \\
&\lesssim |Q|^{1/2}\sum_{i\leq n-1}\Big(\int_{Q_i}{\left\|\Big(\int_{Q}|K_{i,Q}(x,y)|^2 p_Qf(y)p_Q d y\Big)^{\frac12}\right\|}_{L_2(\N)}^2w(x)dx\Big)^{1/2}\Big(\int_{Q_i}w(x)dx\Big)^{1/2}\\
&\lesssim w(Q)^{\frac12}\sum_{i\leq n-1}|Q_i|^{1/2} \Big(\int_{Q}{\big\|p_Qf(y)p_Q \big\|}_{L_1(\N)}\int_{Q_i}|K_{i,Q}(x,y)|^2w(x)dxdy\Big)^{1/2}.
\end{align*}
Plugging the above estimate into \eqref{key-h} and using Lemma \ref{keylemma2} with $p=2$, we conclude that
$$\|H\|_{L_1^w(\M)}\lesssim \sum_{n\geq 1}\sum_{Q\in D_n}(\lambda \nu(p_Q)w(Q))^{\frac12}\|\chi_Qp_nfp_n\|_{L_1^w(\M)}^{\frac 1 2}.$$
Using the fact $\varphi^w(\chi_Q p_n)=\nu(p_Q)w(Q)$ and the Cauchy-Schwartz inequality, we obtain
\begin{equation}\label{last-part}
\begin{aligned}
  {\|H\|}_{L_1^w(\M)}&\lesssim  \sum_{n\geq 1}\sum_{Q\in D_n} \left(\lambda \varphi^w(\chi_Q p_n)\right)^{\frac 1 2}\varphi^w(\chi_Qp_nf)^{\frac 1 2}\\
&\leq \Big(\sum_{n\geq1}\sum_{Q\in D_n}
\lambda\varphi^w(\chi_Qp_n)\Big)^{\frac 1 2}\Big(\sum_{n\geq1}\sum_{Q\in D_n}
\varphi^w(\chi_Qp_nf)\Big)^{\frac 1 2}
\\
&=\big(\lambda\varphi^w({\bf 1}_{\M}-q)\big)^{\frac 1 2}\varphi^w\big(({\bf 1}_{\M}-q)f\big)^{\frac 1 2}\lesssim {\|f\|}_{L_1^w (\M)},
\end{aligned}
\end{equation}
where the last inequality is due to Lemma \ref{Cuculescu}(iv). This completes the proof.
\end{proof}


\section{Weighted estimate for square functions}\label{sec5}
In this section, we provide the proof of Theorem \ref{main-square}. To this end, we first establish the following lemma.

\begin{lemma}\label{keylemma-square}
Let $1\leq p<\infty$ and let a weight $w\in A_1$. Suppose $K=(K_k)_{k\geq1}$ satisfies the size condition and \eqref{Lr-Hc-square} with $r=p r_w'$ $($$r_w'$ is given in Theorem \ref{main-maximal}$)$. Then, for any cube $Q$, we have
\begin{align*}
\sum_{k\geq 1}\left|5^{k+1} Q\right|\left(\frac{1}{\left|5^{k+1}Q\right|} \int_Q\|f(y)\|_{L_1(\N)}\int_{5^{k+1}Q\setminus5^k Q}{\|\vec{K}_Q(x,y)\|}_{\ell_2}^p w(x) d x d y\right)^{\frac{1}{p}}
 \lesssim\left\|\chi_Q f\right\|_{L_1^w(\M)}^{\frac{1}{p}}.
\end{align*}
\end{lemma}
\begin{proof}
For simplicity, we set $Q_k=5^{k+1}Q\setminus5^k Q$. Using
the H\"{o}lder inequality and Lemma \ref{weightprop}(i), we get, for any $y \in Q$,
 \begin{align*}
\int_{Q_k}&{\|\vec{K}_Q(x,y)\|}_{\ell_2}^p w(x) d x\\
&\leqslant\left|5^{k+1} Q\right|\left(\frac{1}{|5^{k+1}Q |} \int_{Q_k}{\|\vec{K}_Q(x,y)\|}_{\ell_2}^{p r_w'}\right)^{\frac{1}{r_w'}}\left(\frac{1}{|5^{k+1} Q|} \int_{Q^k} w(x)^{r_w} d x\right)^{\frac{1}{r_w}}\\
&\lesssim\left|5^{k+1} Q\right|\left(\frac{1}{|5^{k+1} Q|} \int_{Q_k}{\|\vec{K}_Q(x,y)\|}_{\ell_2}^{p r_w^{\prime}}\right)^{\frac{1}{r_w'}}\left(\frac{1}{|5^{k+1} Q|} \int_{5^{k+1} Q} w(x) d x\right)\\
&\lesssim\left|5^{k+1} Q\right|\left(\frac{1}{\left|5^{k+1} Q\right|} \int_{Q_k}{\|\vec{K}_Q(x,y)\|}_{\ell_2}^{p r_w'}\right)^{\frac{1}{r_w'}} w(y),
\end{align*}
where the last inequality is due to $w\in A_1$. We denote by ${\rm (LHS)}$ the left-hand side of the desired inequality. Since $(K_k)_{k\geq 1}$ satisfies \eqref{Lr-Hc-square} with $r=p r_w'$, we have
\begin{align*}
{\rm (LHS)} &\lesssim \sum_{k\geq 1}\left|5^{k+1} Q\right|\left(\int_Q\|f(y)\|_{L_1(\N)}\left(\frac{1}{\left|5^{k+1} Q\right|} \int_{Q_k}{\|\vec{K}_Q(x,y)\|}_{\ell_2}^{p r_w'}\right)^{\frac{1}{r_w'}} w(y)dxdy\right)^{\frac 1 p}\\
&\leq\sum_{k\geq 1}\sup \limits_{y \in Q}\left|5^{k+1} Q\right|\left(\frac{1}{\left|5^{k+1} Q\right|} \int_{Q_k}{\|\vec{K}_Q(x,y)\|}_{\ell_2}^{p r_w'}dx\right)^{\frac{1}{pr_w'}}\left(\int_Q\|f(y)\|_{L_1(\N)}w(y)dy\right)^{\frac 1 p}\\
&\lesssim \left(\int_Q\|f(y)\|_{L_1(\N)}w(y)dy\right)^{\frac 1 p}=\left\|\chi_Q f\right\|_{L_1^w(\M)}^{\frac{1}{p}},
\end{align*}
which completes the proof.
\end{proof}

Now, we are in a position to prove Theorem \ref{main-square}.

 \begin{proof}[Proof of Theorem \ref{main-square}]
 Without loss of generality, we may assume that $f\geq0$. 	
 Set
 $$\widetilde{T}f = \sum_{k\geq1}\varepsilon_kT_k(f),$$
 where $(\varepsilon_k)_{k\geq 1}$ is a sequence of Rademacher functions defined on $(0,1)$.
 By \eqref{khin1}, it is enough to prove
 $$\|\widetilde{T}f\|_{L^w_{1,\infty}(\mathcal A)}=\Big\|\sum_{k\geq1}\varepsilon_kT_k(f)\Big\|_{L_{1,\infty}(\A^w)}\lesssim \|f\|_{L^w_1(\M)},$$
 	where $\A^w=L_{\infty}(0,1)\overline{\otimes}\M$ equipped with the trace $\widetilde{\varphi^w}=\int_0^1\otimes \varphi^w$.
 	To this end, for fixed $\lambda>0$, consider the Calder\'{o}n-Zygmund decomposition \eqref{dec} of $f$:
 	$$f=g+b_{\mathrm{d}}+b_{\mathrm{off}}.$$
Obviously,
 \begin{align*}		&\lambda\widetilde{\varphi^w}(\chi_{(\lambda,\infty)}(|\widetilde{T}(f)|))\\
 &\quad \leq\lambda\widetilde{\varphi^w}(\chi_{(\lambda/3,\infty)}(|\widetilde{T}(g)|))+\lambda\widetilde{\varphi^w}(\chi_{(\lambda/3,\infty)}(|\widetilde{T}(b_{\mathrm{d}})|))+\lambda\widetilde{\varphi^w}(\chi_{(\lambda/3,\infty)}(|\widetilde{T}(b_{\mathrm{off}})|))\\
 		&\quad  =:\mathrm{I}+\mathrm{II}+\mathrm{III}.
 	\end{align*}
It is sufficient for us to estimate $\mathrm{I}$, $\mathrm{II}$ and $\mathrm{III}$ separately.

 The estimate of $\mathrm{I}$ is easy. Indeed, by the Chebyshev inequality and the assumption that $T$ is bounded from $L_2^w(\M)$ to $L_2^w(\M;\ell_2^{cr})$, we have
 \begin{align*}
 	\mathrm{I}&\lesssim \frac{1}{\lambda} \|\widetilde{T}(g)\|_{L_2^w(\A)}^2=\frac{1}{\lambda} \|T(g)\|_{L_2^w(\M;\ell_2^{cr})}^2\lesssim \frac{1}{\lambda} \|g\|_{L_2^w(\M)}^2\lesssim \frac{1}{\lambda}\lambda \|f\|_{L_1^w(\M)}=\|f\|_{L_1^w(\M)},
 \end{align*}
where the last inequality is due to Lemma \ref{NewCZ}(i).

We now turn to estimate $\mathrm{II}$. We rewrite 
$$\widetilde{T}(b_{\mathrm{d}})=\sum_{k\geq1}\varepsilon_kT_k(b_{\mathrm{d}})({\bf 1_\M}-\zeta)+\sum_{k\geq1}\varepsilon_k({\bf 1_\M}-\zeta)T_k(b_{\mathrm{d}})\zeta+\sum_{k\geq1}\varepsilon_k\zeta T_k(b_{\mathrm{d}})\zeta.$$
According to Lemma \ref{lem1}(i), we have
 \begin{align*}
 	\mathrm{II}&\leq 2 \lambda \widetilde{\varphi^w} (1_{(0,1)}\otimes ({\bf 1_\M}-\zeta)) +\lambda \widetilde{\varphi^w}\Big(\chi_{(\lambda/3,\infty)}\Big(\Big|\sum_{k\geq1}\varepsilon_k\zeta T_k(b_{\mathrm{d}})\zeta\Big|\Big)\Big)\\
 	&\lesssim \|f\|_{L_1^w(\M)}+\lambda \widetilde{\varphi^w}\Big(\chi_{(\lambda/3,\infty)}\Big(\Big|\sum_{k\geq1}\varepsilon_k\zeta T_k(b_{\mathrm{d}})\zeta\Big|\Big)\Big)\\
 	&=: \|f\|_{L_1^w(\M)}+\mathrm{II}_{\zeta}.
 \end{align*}
Recall that
$$b_{\mathrm{d}}=\sum_{n\geq1}\sum_{Q\in D_n}p_{Q}(f-f_Q)p_Q\chi_Q=:\sum_{n\geq1}\sum_{Q\in D_n}b_{\mathrm{d},n,Q}.$$
By Lemma \ref{lem1}, we have $p_Q\zeta(x)=\zeta(x)p_Q=0$ provided $x\in 5Q$.
Hence, 
$$\zeta T_k(b_{\mathrm{d},n,Q})\zeta=\zeta T_k(b_{\mathrm{d},n,Q})\zeta\chi_{(5Q)^c}.$$
Then the Chebyshev inequality and triangle inequality give us
\begin{align*}
	\mathrm{II}_{\zeta}&\leq 3 \Big\|\sum_{k\geq1}\varepsilon_k\zeta T_k(b_{\mathrm{d}})\zeta\Big\|_{L_1^w(L_{\infty}(0,1)\overline{\otimes}\N)}\leq 3\sum_{n\geq1}\sum_{Q\in D_n} \Big\|\sum_{k\geq1}\varepsilon_k\zeta T_k(b_{\mathrm{d},n,Q})\zeta\Big\|_{L_1^w(L_{\infty}(0,1)\overline{\otimes}\N)}\\
	&=3\sum_{n\geq1}\sum_{Q\in D_n} \int_{(5Q)^c}  \Big\|\sum_{k\geq1}\varepsilon_k\zeta(x) T_k(b_{\mathrm{d},n,Q})(x)\zeta(x)\Big\|_{L_1(L_{\infty}(0,1)\overline{\otimes}\N)}w(x)dx\\
	&\leq 3\sum_{n\geq1}\sum_{Q\in D_n} \int_{(5Q)^c}  \Big\|\sum_{k\geq1}\varepsilon_k T_k(b_{\mathrm{d},n,Q})(x)\Big\|_{L_1(L_{\infty}(0,1)\overline{\otimes}\N)}w(x)dx\\
&=: 3\sum_{n\geq1}\sum_{Q\in D_n} H_Q.
\end{align*}
From Lemma \ref{NewCZ}(iii), we have
$$T_k(b_{\mathrm{d},n,Q})(x)=\int_{Q} b_{\mathrm{d},n,Q}(y) \big(K_k(x,y)-K_k(x,c_Q)\big)dy=:\int_{Q} b_{\mathrm{d},n,Q}(y) K_{k,Q}(x,y)dy, $$
where $c_Q$ is the centre of $Q$. Furthermore,
\begin{align*}
	&\Big\|\sum_{k\geq1}\varepsilon_k  T_k(b_{\mathrm{d},n,Q})(x) \Big\|_{L_1(L_{\infty}(0,1)\overline{\otimes}\N)}\\
	&\leq \int_{Q} \Big\|b_{\mathrm{d},n,Q}(y)\sum_{k\geq1}\varepsilon_k  K_{k,Q}(x,y) \Big\|_{L_1(L_{\infty}(0,1)\overline{\otimes}\N)} dy\\
	&=\int_{Q} \|b_{\mathrm{d},n,Q}(y)\|_{L_1(\N)}\Big\|\sum_{k\geq1}\varepsilon_k  K_{k,Q}(x,y) \Big\|_{L_1(L_{\infty}(0,1))} dy\\
	&\lesssim \int_{Q} \|b_{\mathrm{d},n,Q}(y)\|_{L_1(\N)}\Big(\sum_{k\geq1}  |K_{k,Q}(x,y)|^{2}\Big)^{1/2} dy,
\end{align*}
where the last inequality is due to the classical Khintchine inequality. Applying Lemma \ref{keylemma-square} with $p=1$, we obtain

\begin{align*}
H_Q&\leq \sum_{j\geq 1}\int_{5^{j+1}Q\setminus5^j Q}\int_{Q} \|b_{\mathrm{d},n,Q}(y)\|_{L_1(\N)}\Big(\sum_{k\geq1}  |K_{k,Q}(x,y)|^{2}\Big)^{1/2} dyw(x)dx\\
&\lesssim\|b_{\mathrm{d},n,Q}\|_{L_1^w(\M)}.
\end{align*}
We conclude from the above argument and Lemma \ref{NewCZ}(ii) that
$$\mathrm{II}_{\zeta}\lesssim \sum_{n\geq1}\sum_{Q\in D_n}\|b_{\mathrm{d},n,Q}\|_{L_1^w(\M)} \lesssim \|f\|_{L_1^w(\M)}.$$

Finally, we estimate $\mathrm{III}$. Let $b_{\mathrm{off},n,Q}=p_Qfq_Q$. Then
$$b_\mathrm{off}=\sum_{n\geq1}\sum_{Q\in D_n} b_{\mathrm{off},n,Q}+b_{\mathrm{off},n,Q}^*.$$
By similar argument used as in the estimate of $\mathrm{II}$, we obtain
\begin{align*}
	\mathrm{III}\lesssim 2\|f\|_{L_1(\M)}+ \mathrm{III}_{\zeta},
\end{align*}
where $\mathrm{III}_{\zeta}$ is estimated as follows:
\begin{equation}\label{ineq1}
\begin{split}
	\mathrm{III}_{\zeta}&= \lambda \widetilde{\varphi}\Big(\chi_{(\lambda/3,\infty)}\Big(\Big|\sum_{k\geq1}\varepsilon_k\zeta T_k(b_{\mathrm{off}})\zeta\Big|\Big)\Big)
	\leq 3 \Big\|\sum_{k\geq1}\varepsilon_k\zeta T_k(b_{\mathrm{off}})\zeta\Big\|_{L_1^w(L_{\infty}(0,1)\overline{\otimes}\N)}\\
	&\leq 3\sum_{n\geq1}\sum_{Q\in D_n} \Big\|\sum_{k\geq1}\varepsilon_k\zeta T_k(b_{b_{\mathrm{off},n,Q}+b_{\mathrm{off},n,Q}^*})\zeta\Big\|_{L_1^w(L_{\infty}(0,1)\overline{\otimes}\N)}\\
	&\leq 6\sum_{n\geq1}\sum_{Q\in D_n} \Big\|\sum_{k\geq1}\varepsilon_k\zeta T_k(b_{b_{\mathrm{off},n,Q}})\zeta\Big\|_{L_1^w(L_{\infty}(0,1)\overline{\otimes}\N)}.
\end{split}
\end{equation}
For each fixed $n$ and $Q\in D_n$, we have from Lemma \ref{NewCZ}(iii) that
\begin{equation} \label{ineq2}
  \begin{split}
  &\Big\|\sum_{k\geq1}\varepsilon_k\zeta T_k(b_{{\mathrm{off},n,Q}})\zeta\Big\|_{L_1^w(L_{\infty}(0,1)\overline{\otimes}\N)}    \\
  &=   \int_{(5Q)^c} \Big\|\sum_{k\geq1}\varepsilon_k T_k(b_{\mathrm{off},n,Q})(x)\Big\|_{L_1(L_{\infty}(0,1)\overline{\otimes}\N)}w(x)dx\\
	&= \int_{(5Q)^c} \Big\|\int_Q p_Qf(y)q_Q \sum_{k\geq1}\varepsilon_k K_{k,Q}(x,y)dy\Big\|_{L_1(L_{\infty}(0,1)\overline{\otimes}\N)}w(x)dx.
  \end{split}
  \end{equation}
By the H\"{o}lder inequality stated in \cite[Proposition 1.1]{Me2007}, we have
\begin{align*}
&	\Big\|\int_Q p_Qf(y)q_Q \sum_{k\geq1}\varepsilon_k K_{k,Q}(x,y)dy\Big\|_{L_1(L_{\infty}(0,1)\overline{\otimes}\N)}\\
&\leq
	\Big\|\Big(\int_Q p_Qf(y)p_Q K_{\varepsilon,Q}(x,y)^2dy\Big)^{\frac 1 2}\Big\|_{L_1(L_{\infty}(0,1)\overline{\otimes}\N)} \cdot	\Big\|\Big(\int_Q q_Qf(y)q_Q dy\Big)^{\frac 1 2}\Big\|_{L_{\infty}(L_{\infty}(0,1)\overline{\otimes}\N)}\\
	&\leq \Big\|\Big(\int_Q p_Qf(y)p_Q K_{\varepsilon,Q}(x,y)^2dy\Big)^{\frac 1 2} \Big\|_{L_2(L_{\infty}(0,1)\overline{\otimes}\N)}[\nu(p_Q)]^{1/2} (\lambda|Q|)^{1/2}\\
&=\Big\|\int_Q p_Qf(y)p_Q K_{\varepsilon,Q}(x,y)^2dy \Big\|_{L_1(L_{\infty}(0,1)\overline{\otimes}\N)}^{\frac 1 2} [\nu(p_Q)]^{\frac 1 2} (\lambda|Q|)^{\frac 1 2},
\end{align*}
where $K_{\varepsilon,Q}(x,y)=\sum_{k\geq1}\varepsilon_k K_{k,Q}(x,y)$.  The classical Khintchine inequality gives that
\begin{align*}
&\Big\|\int_Q p_Qf(y)p_Q K_{\varepsilon,Q}(x,y)^2dy\Big\|_{L_1(L_{\infty}(0,1)\overline{\otimes}\N)}\\&\leq \int_Q\Big\| p_Qf(y)p_Q K_{\varepsilon,Q}(x,y)^2\Big\|_{L_1(L_{\infty}(0,1)\overline{\otimes}\N)}dy\\
&\lesssim\int_Q {\|p_Qf(y)p_Q\|}_{L_1(\N)}\Big(\sum_{k\geq1}  |K_{k,Q}(x,y)|^{2}\Big)dy,
\end{align*}
where the first inequality is due to \cite[Proposition 1.2.2]{HvanNVW2016}. Plugging the above estimate into \eqref{ineq2}, we deduce that
\begin{align*}
 &\Big\|\sum_{k\geq1}\varepsilon_k\zeta T_k(b_{{\mathrm{off},n,Q}})\zeta\Big\|_{{L_1^w({L_{\infty}(0,1)\overline{\otimes}\N)}}}\\
 &\leq  \int_{(5Q)^c} (\lambda|Q|)^{\frac 1 2} [\nu(p_Q)]^{\frac 1 2}\left(\int_Q {\|p_Qf(y)p_Q\|}_{L_1(\N)}\Big(\sum_{k\geq1}  |K_{k,Q}(x,y)|^{2}\Big)dy\right)^{\frac 1 2}w(x)dx\\
 &=(\lambda\nu(p_Q))^{\frac 1 2}\sum_{j\geq 1}|Q|^{1/2}\int  _{Q_j}\left(\int_Q{\|p_Qf(y)p_Q\|}_{L_1(\N)}\Big(\sum_{k\geq1}  |K_{k,Q}(x,y)|^{2}\Big)dy\right)^{\frac 1 2}{w(x)}dx,
\end{align*}
where $Q_j=5^{j+1}Q\setminus5^j Q$. By Lemma \ref{weightprop}(ii), for any $j\geq 1$, we have
$$
|Q|\lesssim \big|5^{j+1}Q\big|\frac{w(Q)}{w(5^{j+1}Q)}.
$$
Then, using the Cauchy-Schwarz inequality, we arrive at
\begin{align*}
&\Big\|\sum_{k\geq1}\varepsilon_k\zeta T_k(b_{b_{\mathrm{off},n,Q}})\zeta\Big\|_{L_1^w(L_{\infty}(0,1)\overline{\otimes}\N)}\\
&\lesssim L_Q\sum_{j\geq 1}\left(\frac{\big|5^{j+1}Q\big|}{w(5^{j+1}Q)}\int  _{Q_j}\int_Q{\|p_Qf(y)p_Q\|}_{L_1(\N)}{\|\vec{K}_Q(x,y)\|}_{\ell_2}^2 dyw(x)dx\right)^{\frac12}
\left(\int  _{Q_j} w(x)dx\right)^{\frac12}\\
&\leq L_Q\sum_{j\geq 1}|5^{j+1}Q|^{\frac12}\left(\int_{Q_j}\int_Q{\|p_Qf(y)p_Q\|}_{L_1(\N)}{\|\vec{K}_Q(x,y)\|}_{\ell_2}^2{w(x)}dydx\right)^{\frac 1 2},
\end{align*}
where $L_	Q=(\lambda \nu(p_Q)w(Q))^{\frac 1 2}$.
Applying Lemma \ref{keylemma-square} with $p=2$, we get
\begin{equation}\label{ineq3}
\Big\|\sum_{k\geq1}\varepsilon_k\zeta T_k(b_{b_{\mathrm{off},n,Q}})\zeta\Big\|_{L_1^w((L_{\infty}(0,1)\overline{\otimes}\N))}\lesssim \left\|\chi_Q p_nfp_n\right\|_{L_1^w(\M)}^{\frac{1}{2}}.
\end{equation}
Note that $\varphi^w(\chi_Q p_n)=\nu(p_Q)w(Q)$. Combining \eqref{ineq1}, \eqref{ineq3} with Lemma \ref{NewCZ} and following the similar proof as \eqref{last-part}, we obtain the desired estimate
$$
\mathrm{III}_{\zeta}\lesssim {\|f\|}_{L_1^w (\M)}.
$$
The proof is finished.
 \end{proof}

\medskip


\section{Weighted estimate in Hardy spaces}\label{sec6}

This section is devoted to proving Theorem \ref{HL1}. To this end, we establish the atomic decomposition theorem for martingale Hardy space $H_{1w}^c(\M)$, which can be regarded as a weighted version of \cite[Corollary 3.12]{CRX2023}. In what follows, we denote by $H_{1w}^{0,c}(\M)$ (resp. $h_{1w}^{0,c}(\M)$), the closure of the set $\{f\in L_1(\M)\cap \M:\E_1(f)=0\}$ in ${H_{1w}^c(\M)}$ (resp. $h_{1w}^{c}(\M)$).

\subsection{Atomic decomposition of weighted Hardy spaces}\label{atomic-hardy}

We firstly introduce the various concepts of noncommutative atoms associated with a weight $w$. Note that if we take $w\equiv 1$, then the following atoms are reduced to the noncommutative atoms introduced in \cite{BCPY2010,MP2011}.

\begin{definition}\label{def-atom}
	An element $a\in L_2^w(\M)$ is said to be a $(1,2)_{c}^w$-atom with respect to the filtration $(\M_k)_{k\geq1}$ and a weight $w$ if there exist $k\geq1$ and a projection $e\in \M_k$ such that
	\begin{enumerate}[{\rm (1)}]
		\item $a=ae$;
		\item $\mathbb{E}_k(a)=0$;
		\item $\|a\|_{L_2^w(\M)}\leq \varphi(ew)^{-1/2}$.
	\end{enumerate}
Row
atoms are defined to satisfy $a=ea$ instead.
\end{definition}

\begin{definition}\label{def-alge-atom}
An element $z\in L_1^w(\M)$ is said to be an algebraic $h_{1w}^c$-atom with respect to the filtration $(\M_k)_{k\geq1}$ and a weight $w$ if there exist a factorization $z=\sum_{k\geq1}a_kb_k$, with $a_k$ and $b_k$ satisfying
	\begin{enumerate}[{\rm (1)}]
			\item $\E_k(a_k)=0$ for all $k\geq 1$, $\sum_{k\geq 1}\|a_k\|_{L_2^w(\M)}^2\leq 1$;
		\item $b_k\in L_2^w(\M_k)$ for all $k\geq 1$, $ \Big\|\big(\sum_{k\geq 1}|b_k|^2\big)^{1/2}\Big\|_{L_2^w(\M)}\leq 1$.	\end{enumerate}
The notion of algebraic $h_{1w}^r$-atoms can be obtained by replacing the factorizations above by $z=\sum_{k\geq1}b_ka_k$.
\end{definition}

Due to the special structure of the dyadic filtration $\{\mathcal M_n\}_{n\geq 1}$ (see \eqref{def-filtration}), we have the following useful properties of weighted $(1,2)_{c}^w$-atoms.

\begin{remark}\label{use-rem} Assume that $a$ is an $(1,2)_{c}^w$-atom associated with $k\geq1$ and $e\in \M_k$. From the above definition, we have the following facts.
	
	\begin{enumerate}[{\rm (i)}]
		\item We have
		\begin{equation}\label{aL1}
			\|a\|_{L_1^w(\M)}=\|aew\|_{L_1{(\M)}}\leq \|a\|_{L_2^w(\M)}\varphi(ew)^{1/2}\leq 1.
		\end{equation}
		\item Since $e\in \M_k$, we can write
		$$e=\sum_{Q\in D_k}p_Q\chi_{Q},$$
		where $p_Q\in \mathcal{P}(\N)$ for each $Q\in D_k$. Moreover, we have
		\begin{equation}\label{a}
		a=\sum_{Q\in D_k} a\chi_{Q} =\sum_{Q\in D_k} ae\chi_{Q}=\sum_{Q\in D_k} ap_Q\chi_{Q}.
		\end{equation}
		\item Consider
	$$\eta =\bigvee_{Q\in D_k} p_Q\chi_{5Q},\quad \eta^\perp=\1-\eta,$$
	where $5Q$ is the interval with the same centre as $Q$ with $|5Q|=5^d|Q|$. Assume $w\in A_1$. Then it follows that
	\begin{equation}\label{eta}
		\varphi(\eta w)\lesssim\varphi(ew).
	\end{equation}
	Indeed, by Lemma \ref{weightprop}(ii), we know that
	\begin{align*}
	\varphi(\eta w)&\leq\sum_{Q\in D_k}\varphi(p_Q\chi_{5Q}w)=\sum_{Q\in D_k}\nu(p_Q)\int_{5Q}w(x)dx\\
	&\lesssim \sum_{Q\in D_k}\nu(p_Q) \int_{Q}w(x)dx=\sum_{Q\in D_k}\varphi(p_Q\chi_Qw)=\varphi(e w).
	\end{align*}
\item According to Definition \ref{def-atom} (2),
$$\E_k(a)=\sum_{Q\in D_k}\frac{1}{|Q|}\int_Qa(y)dy\chi_Q=0,$$
which further implies
\begin{equation}\label{aQ}
	\int_Qa(y)dy=0,\quad \forall Q\in D_k.
\end{equation}
	\end{enumerate}
\end{remark}

We now introduce the atomic Hardy spaces $ H_{1,\rm{at}}^{cw}(\M)$, $H_{1,\rm{aa}}^{cw}(\M)$ associated with weights.
\begin{definition}\label{def-atomic-spaces}
Let $ H_{1,\rm{at}}^{cw}(\M)$ be the space of all operators $f\in H_{1w}^c(\M)$	which can be written as
	\begin{equation}\label{decom}
	f=\sum_k\lambda_k a_k,
	\end{equation}
	where for each $k$, {$a_k$ is either a $(1,2)_{c}^w$-atom or an element of the unit ball of $L_1^w(\M_1)$}, and $(\lambda_k)_k\subset \mathbb{C}$ satisfying $\sum_k|\lambda_k|<\infty$. For $f\in H_{1,\rm{at}}^{cw}(\M)$ we define
	$$
	\|f\|_{H_{1,\rm{at}}^{cw}(\M)}=\inf \sum_{k}|\lambda_k|,
	$$
where the infimum is taken over all the decompositions of $f$ as in \eqref{decom}. Replacing $(1,2)_{c}^w$-atoms with algebraic $h_{1w}^c$-atoms in the above definitions, we obtain the algebraic-atom spaces $H_{1,\rm{aa}}^{cw}(\M)$.
\end{definition}

\begin{remark}\label{indspaces}
These two types of atomic Hardy spaces are indeed identical with equivalent norms. Namely,
\begin{equation}\label{equiv-atomspaces}
 H_{1,\rm{at}}^{cw}(\M)=H_{1,\rm{aa}}^{cw}(\M).
\end{equation}
We include a proof of \eqref{equiv-atomspaces} in Appendix \ref{appendix}.
\end{remark}

Next, we establish the atomic decomposition theorem for weighted Hardy spaces ${H_{1w}^{c}(\M)}$.

\begin{theorem}\label{atomic}
Given a weight $w\in A_2$, we have
	$$
{H_{1w}^{0,c}(\M)}={H_{1,\rm{at}}^{cw}(\M)}
$$
with equivalent norms. Similar result holds for the row version.
	
\end{theorem}

To prove Theorem \ref{atomic}, we first introduce the UMD spaces. Let $(\Omega,\F,(\F_n)_{n\geq 1},\mu)$ be a filtered $\sigma$-finite measure space. A Banach space $\mathbb{X}$ is a UMD (Unconditional for Martingale Differences) space if for some (equivalently, all) $p\in (1,\infty)$, there is a constant $C_{p}$ such that for any martingale difference $(df_n)_{n\geq 1}$ with values in $\mathbb{X}$,
\begin{equation*}
\Big\|\sum_{n=1}^k r_ndf_n\Big\|_{L_p(\Omega;\mathbb{X})}\leq C_{p}\Big\|\sum_{n=1}^k df_n\Big\|_{L_p(\Omega;\mathbb{X})},\qquad k\geq 1,
\end{equation*}
where $(r_n)_n$ is a deterministic sequence with values in $\{1, -1\}$.

Let $\mathbb{X}$ be a UMD space and $w$ be an $A_p$ weight on $\Omega$. The following result is contained in \cite{La2017}:
if $(df_n)_{n\geq 1}$ is a martingale difference with values in $\mathbb{X}$, $(r_n)_{n\geq 1}$ is a predictable sequence of signs, then we have
\begin{equation}\label{w_transform}
 \Big\|\sum_{n=1}^k r_ndf_n\Big\|_{L_p^w(\Omega;\mathbb{X})}\lesssim [w]_{A_p}^{\max\{1/(p-1),1\}}\Big\|\sum_{n=1}^k df_n\Big\|_{L_p^w(\Omega;\mathbb{X})}, k\geq 1.
\end{equation}
Here $\|f\|_{L_p^w(\Omega;\mathbb{X})}=\left(\int_\Omega \|f\|_\mathbb{X}^pw d\mu\right)^{1/p}$. 

\begin{lemma}\label{L2}
Given a weight $w\in A_2$ and $f\in L_2^w(\M)$, the following inequalities
\begin{equation*}
\big\|(df_n)_{n\geq l}\big\|_{L_2^w(\M;\ell_2)}^2\approx
\Big\|\sum_{n\geq l}df_n\Big\|_{L_2^w(\M)}^2
\end{equation*}
holds for every integer $l\geq 1$.
\end{lemma}
\begin{proof}
{Note that $L_2(\mathcal{N})$ is a UMD space. It follows from \eqref{w_transform} that
\begin{align*}
\Big\|\sum_{n\geq l}  r_ndf_n\Big\|_{L_2(\A^w)}^2&=\int_0^1 \Big\|\sum_{n\geq l} r_ndf_n\Big\|_{L_2^w(\R^d;{L_2(\N)})}^2dt \\&
\lesssim[w]_{A_2}^2
\int_0^1 \Big\|\sum_{n\geq l} df_n\Big\|_{L_2^w(\R^d;{L_2(\N)})}^2dt\\&
=\Big\|\sum_{n \geq l} df_n\Big\|_{L_2(\A^w)}^2,
\end{align*}
}
where $\mathcal{A}^w=L_{\infty}(0,1) \bar{\otimes} \mathcal{M}$ equipped with the trace $\widetilde{\varphi^w}=\int_0^1 \otimes \varphi^w$ and {$(r_n(t))_n$} is a deterministic sequence with values in $\{1, -1\}$. Then, combining with \eqref{khin2}, we have
\begin{align*}
\big\|(df_n)_{n\geq l}\big\|_{L_2^w(\M;\ell_2)}^2&=\Big\|\sum_{n\geq l}  r_ndf_n\Big\|_{L_2(\A^w)}^2\lesssim \Big\|\sum_{n \geq l}  df_n\Big\|_{L_2(\A^w)}^2\\&
=\Big\|\sum_{n \geq l}  r_n^2df_n\Big\|_{L_2(\A^w)}^2
\lesssim \Big\|\sum_{k\geq l}  r_ndf_n\Big\|_{L_2(\A^w)}^2\\&
=\Big\|(df_n)_{n\geq l}\Big\|_{L_2^w(\M;\ell_2)}^2.
\end{align*}
The assertion follows.
\end{proof}

\begin{proof}[Proof of Theorem \ref{atomic}]
We only consider the column case here since the proof of the row one is identical. Assume that $a$ is an $(1,2)_c^w$ atom, and $e\in \M_n$ is the projection associated with $a$ in Definition \ref{def-atom}. Obviously, $\E_1(a)=\E_1(\E_n(a))=0$. Let $a_k=\E_k(a)$, $k\geq 1$. It was shown in \cite[Proposition 2.2]{BCPY2010} that
$$
eS_c(a)=S_c(a)=S_c(a)e.
$$
Therefore, by the H\"{o}lder inequality and Lemma \ref{L2}, we get
\begin{align*}
\|a\|_{H_{1w}^c(\mathcal{M})}&=\varphi(S_c(a)ew)\leq
\|S_c(a)\|_{L_2^w(\M)}\|e\|_{L_2^w(\M)}\\&\approx\|a\|_{L_2^w(\M)}\varphi(ew)^{1/2}\leq 1,
\end{align*}
which implies that
$$
 H_{1,\rm{at}}^{cw}(\M)\subset H_{1w}^{0,c}(\M).
$$

Conversely, according to Lemma \ref{use-lem}, it suffices to prove the following inequality
\begin{equation}\label{atomic-converse}
\|f\|_{H_{1,\rm{at}}^{cw}(\mathcal{M})}\lesssim \|f\|_{h_{1w}^{c}(\mathcal{M})}
\end{equation}
holds for every martingale $f=(f_n)_{n\geq 1}$ in $h_{1w}^{0,c}(\M)$. Without loss of generality, we may assume that $s_{c,n}(f)$ are invertible with bounded inverses. For the sake of convenience, we simply write $s_n$ for $s_{c,n}(f)$. Then, we rewrite

$$
\begin{aligned}
f & =\sum_{n \geq 2} d f_n s_n^{-1} s_n=\sum_{n \geq 2} d f_n s_n^{-1}
\Big(\sum_{2 \leq j \leq n}\left(s_j-s_{j-1}\right)\Big) \\
& = \sum_{n\geq 2}\sum_{1\leq l\leq n-1}df_ns_n^{-1}(s_{l+1}-s_{l})=\sum_{l\geq 1}\sum_{n\geq l+1}df_ns_n^{-1}(s_{l+1}-s_l),
\end{aligned}
$$
with the convention that $s_1=0$. Set
$$
\left\{\begin{aligned}
&\alpha_l  :=\sum_{n \geq l+1} d f_n s_n^{-1}\left(s_{l+1}-s_l\right)^{1 / 2}\\ 
&\beta_l  :=\left(s_{l+1}-s_l\right)^{1 / 2}, \end{aligned}\right.
$$
for each $l\geq 1$. Then, $f$ can be expressed as
$$
f=\sum_{l \geq 1} \alpha_l \beta_l .
$$
Due to the predictability of $(s_n)_{n\geq 1}$, it can be easily confirmed that $\E_l(\alpha_l)=0$ and $\beta_l\in L_2^w(\M_l)$ for each $l\geq 1$.
Next, we estimate the sequences $(\alpha_l)_{l\geq 1}$ and $(\beta_l)_{l\geq 1}$ respectively.

To deal with the sequence $(\alpha_l)_{l\geq 1}$. Note that for any fixed $l$ and all $n\geq l+1$, we have
$$\E_{l-1}\big( d f_n s_n^{-1}\left(s_{l+1}-s_l\right)^{1 / 2}\big)=0.$$
For each $l\geq1$, set $\gamma_n=d f_n s_n^{-1}\left(s_{l+1}-s_l\right)^{1 / 2}$. Then the sequence $ (\gamma_n)_{n\geq l+1}$ can be regarded as a sequence of martingale differences.
It follows from Lemma \ref{L2} that
\begin{align*}
{\|\alpha_l\|}_{L_2^w(\M)}^2&=\Big\|\sum_{n\geq l+1}d f_n s_n^{-1}\left(s_{l+1}-s_l\right)^{1 / 2}\Big\|^{2}_{L_2^w(\M)}\\&
=\Big\|\sum_{n \geq l+1} \gamma_n\Big\|_{L_2^w(\M)}^2
\lesssim\sum_{n\geq l+1}\|\gamma_n\|_{L_2^w(\M)}^2.
\end{align*}
Then, we arrive at
$$
\begin{aligned}
\sum_{l\geq 1}{\|\alpha_l\|}_{L_2^w(\M)}^2&\lesssim
\sum_{l\geq 1}\sum_{n \geq l+1}\Big\| d f_n s_n^{-1}\left(s_{l+1}-s_l\right)^{1 / 2}\Big\|_{L_2^w(\M)}^2\\
&=\sum_{l\geq 1}\sum_{n \geq l+1} \varphi\big(|df_n|^2s_n^{-1} (s_{l+1}-s_l)s_n^{-1}w\big)\\
&=\sum_{n\geq2}\sum_{l=1}^{n-1}
\varphi^w\big(|df_n|^2s_n^{-1} (s_{l+1}-s_l)s_n^{-1}\big)\\
&=\sum_{n\geq 2}\varphi^w\big(s_n^{-1}(s_n^2-s_{n-1}^2)\big)\lesssim
\sum_{n\geq 2}\varphi^w(s_n-s_{n-1})={\|f\|}_{h^c_{1w}(\mathcal{M})},
\end{aligned}
$$
where the last inequality is due to Lemma \ref{usefullem}. Clearly, we have
$$
\begin{aligned}
\Big\|\Big(\sum_{l\geq 1}|\beta_l|^2\Big)^{1/2}\Big\|_{L_2^w(\M)}^2&=
\varphi\big((\sum_{l\geq 1}|\beta_l|^2)w\big)\\
&=\sum_{l\geq 1}\varphi\big(\left(s_{l+1}-s_l\right)w\big)={\|f\|}_{h_{1w}^c(\mathcal{M})}.
\end{aligned}
$$
This leads to
\begin{equation*}\label{beta}
\Big\|\Big(\sum_{l\geq 1}|\beta_l|^2\Big)^{1/2}\Big\|_{L_2^w(\M)}^2\lesssim{\|f\|}_{h_{1w}^c(\mathcal{M})}.
\end{equation*}
Thus, we conclude that
$$
\Big(\sum_{l\geq 1}{\|\alpha_l\|}_{L_2^w(\M)}^2\Big)^{1/2}\Big\|\Big(\sum_{l\geq 1}|\beta_l|^2\Big)^{1/2}\Big\|_{L_2^w(\M)}\lesssim {\|f\|}_{h_{1w}^c(\mathcal{M})}.
$$
This means that there exists a proper constant $\lambda=C{\|f\|}_{h_{1w}^c(\mathcal{M})}$ ($C>0$ and it only depends on the weight $w$) such that $g=\lambda^{-1}f$ is an algebraic $h_{1w}^c$-atom. Therefore, according to \eqref{equiv-atomspaces}, we conclude \eqref{atomic-converse} and finish the proof.
\end{proof}

\subsection{Proof of Theorem \ref{HL1}} Applying the atomic decomposition theorem established in the previous subsection, we prove Theorem \ref{HL1}.
\begin{proof}[Proof of Theorem \ref{HL1}]
We only prove the column case here as the row version can be shown identically. According to Theorem \ref{atomic}, it suffices to show that
\begin{equation}\label{Ta}
\|Ta\|_{L_1^w(\M;\ell_2^c)}:=\Big\|\sum_{n\geq1} T_n(a)\otimes e_{n,1}\Big\|_{L_1^w(\M\overline{\otimes} B(\ell_2))}\lesssim 1
\end{equation}
holds for each column $(1,2)_{c}^w$-atom $a$. Without loss of generality, we may assume that the given atom $a$ associated with $k$ and $e\in \mathcal{P}(\M_k)$.
As mentioned in Remark \ref{use-rem}, we may write
\begin{equation}\label{e}
e=\sum_{Q\in D_k}p_Q\chi_Q,\quad p_Q\in \mathcal{P}(\N).
\end{equation}
and
$$T_n(a)=T_n(a)\eta+T_n(a)\eta^{\perp},$$
where
$$\eta =\bigvee_{Q\in D_k} p_Q\chi_{5Q},\quad \eta^\perp=\1_\M-\eta.$$
Then
$$
\|Ta\|_{L_1^w(\M;\ell_2^c)}\leq \mathrm{I}+\mathrm{II},
$$
where
$$\mathrm{I}=\Big\|\sum_{n\geq1} T_n(a)\eta\otimes e_{n,1}\Big\|_{L_1^w(\M\overline{\otimes} B(\ell_2))},\quad
\mathrm{II}=\Big\|\sum_{n\geq1} T_n(a)\eta^{\perp}\otimes e_{n,1}\Big\|_{L_1^w(\M\overline{\otimes} B(\ell_2))}.$$
We now show $\mathrm{I}\lesssim 1$ and $\mathrm{II}\lesssim 1$, respectively.

We first estimate $\mathrm{I}$. Note that 
$$
\begin{aligned}
\mathrm{I}&=\Big\|\sum_{n\geq1} T_n(a)\eta\otimes e_{n,1}\Big\|_{L_1^w(\M\overline{\otimes} B(\ell_2))}\\
&=\Big\|\Big(\sum_{n\geq1} \eta |T_n(a)|^2 \eta\Big)^{1/2}\Big\|_{L_1^w(\M)}=\Big\|\Big(\sum_{n\geq1} \eta |T_n(a)|^2 \eta\Big)^{1/2}\eta\Big\|_{L_1^w(\M)}\\
&\leq \Big\|\Big(\sum_{n\geq1} \eta |T_n(a)|^2 \eta\Big)^{1/2}\Big\|_{L_2^w(\M)}\cdot \|\eta\|_{L_2^w(\M)}.
\end{aligned}
$$
According to the assumption that $T:L_2^w(\M)\to L_2^w(\M;\ell_2^{c})$ is bounded, we know
$$
\begin{aligned}
\Big\|\Big(\sum_{n\geq1} \eta |T_n(a)|^2 \eta\Big)^{1/2}\Big\|_{L_2^w(\M)}&=\Big\|\eta\Big(\sum_{n\geq1}  |T_n(a)|^2\Big) \eta\Big\|_{L_1^w(\M)}^{1/2}\\
&\leq \Big\|  \sum_{n\geq1}  |T_n(a)|^2 \Big\|_{L_1^w(\M)}^{1/2}\lesssim \|a\|_{L_2^w(\M)}\leq \varphi(ew)^{-1/2},
\end{aligned}
$$
where we used the fact $\|a\|_{L_2^w(\M)}\leq \varphi(ew)^{-1/2}$ as in Definition \ref{def-atom}.
Then, by \eqref{eta}, we obtain
$$\mathrm{I}\lesssim \varphi(ew)^{-1/2}\|\eta \|_{L_2^w(\M)}\lesssim 1.$$

We now turn to estimate $\mathrm{II}$.
We write
$$
a=\sum_{Q\in D_k}a\chi_Q=\sum_{Q\in D_k}a\cdot p_Q\chi_Q.
$$
Then
$$
\begin{aligned}
\mathrm{II}&=\Big\|\sum_{n\geq1} T_n(a)\eta^{\perp}\otimes e_{n,1}\Big\|_{L_1^w(\M\overline{\otimes} B(\ell_2))}\\
&\leq \sum_{Q\in D_k}\Big\|\sum_{n\geq1} T_n(a\chi_Q)\eta^{\perp}\otimes e_{n,1}\Big\|_{L_1^w(\M\overline{\otimes} B(\ell_2))}\\
&=:\sum_{Q\in D_k} \mathrm{II}_Q.
\end{aligned}
$$
We claim that for each fixed $Q\in D_k$,
\begin{equation}\label{claim1}
\mathrm{II}_Q\lesssim \|a\chi_{Q}\|_{L_1^w(\M)}.
\end{equation}
Obviously, if the claim is verified, then by \eqref{aL1}, we arrive at
$$\mathrm{II}\lesssim \sum_{Q\in D_k} \|a\chi_{Q}\|_{L_1^w(\M)}= \|a\|_{L_1^w(\M)}\leq 1.$$
It remains to check the claim \eqref{claim1}.  Fix $Q\in D_k$. Then
\begin{align}\label{IIQ}
\mathrm{II}_Q&=\Big\|\Big(\sum_{n\geq1}\eta^{\perp}| T_n(a\chi_Q)|^2\eta^{\perp}\Big)^{1/2}\Big\|_{L_1^w(\M)}\nonumber\\
&=\Big\|\Big(\sum_{n\geq1}\eta^{\perp}|T_n(a\chi_Q)|^2\eta^{\perp}\Big)^{1/2}\chi_{\R^d\setminus 5Q}\Big\|_{L_1^w(\M)}=:\mathrm{II}_{Q,2}.
\end{align}
To see the second equality, noting that by \eqref{a}, we have $a\chi_Q=a\cdot p_Q\chi_Q$. This implies
$$|T_n(a\chi_Q)|^2=p_Q|T_n(a\chi_Q)|^2p_Q.$$
Clearly $\eta(x)\geq p_Q$ provided $x\in 5Q$. Then, for $x\in 5Q$, $\eta^{\perp}(x)\leq p_Q^{\perp}$, and \eqref{IIQ} follows. Thus, 
\begin{align*}
\mathrm{II}_{Q,2}&=\int_{\R^d\setminus 5Q}\otimes\nu\otimes \mathrm{tr} \Big|\sum_{n\geq1} \int_QK_n(x,y)a(y)dy\otimes e_{n,1}\Big|w(x)dx\\
&=\int_{\R^d\setminus 5Q}\otimes\nu\otimes \mathrm{tr} \Big|\sum_{n\geq1} \int_Q[K_n(x,y)-K_n(x,c_Q)]a(y)dy\otimes e_{n,1}\Big|w(x)dx\\
&=\int_{\R^d\setminus 5Q} \Big\|\int_Q \sum_{n\geq1}K_{n,Q}(x,y)a(y)\otimes e_{n,1}dy\Big\|_{L_1(\N\overline{\otimes}B(\ell_2))}w(x)dx\\
&\leq \int_{\R^d\setminus 5Q} \int_Q\Big\| \sum_{n\geq1}K_{n,Q}(x,y)a(y)\otimes e_{n,1}\Big\|_{L_1(\N\overline{\otimes}B(\ell_2))}dyw(x)dx\\
&=\int_{Q} \|a(y)\|_{L_1(\N)}\int_{\R^d\setminus 5Q} \Big(\sum_{n\geq1}|K_{n,Q}(x,y)|^2\Big)^{1/2}w(x)dx dy\\
&\lesssim \|a\chi_{Q}\|_{L_1^w(\M)},
\end{align*}
where $K_{n,Q}(x,y)=K_n(x,y)-K_n(x,c_Q)$ for each $n\geq 1$, $c_Q$ is the centre of the fixed cube $Q$, and $\mathrm{tr}$ is the standard trace on $B(\ell_2)$. The first inequality above is referred to \cite[Proposition 1.2.2]{HvanNVW2016}, and the last inequality follows from Lemma \ref{keylemma-square} with $p=1$ as the kernel $\vec{K}$ satisfies \eqref{Lr-Hc-square} with $r=r_w'$. The claim \eqref{claim1} is verified.

Combining the estimates of $\mathrm{I}$ and $\mathrm{II}$, we get \eqref{Ta}, and hence the proof is finished.
\end{proof}

\begin{remark}
It is natural to ask whether the martingale Hardy space $H_1^w(\M)$ in Theorem \ref{HL1} can be replaced by the weighted Hardy space $H_1^w(\R,\M)$.  This is not clear to us at this moment. 
\end{remark}

\medskip
\noindent{\bf Acknowledgement }
The authors would like to thank Professor Tao Mei for his valuable discussions and explanations at the time of this writing.

\appendix
\section{Proof of Remark \ref{indspaces}}\label{appendix}
To verify the equivalence of two types of weighted atomic Hardy spaces introduced in Section \ref{atomic-hardy}, we need to introduce another type of weighted atomic spaces, which can be considered as weighted version of the crude type atoms introduced in \cite[Definition 4.1]{HM2012}.
\begin{definition}\label{def-crude-atom}
An element $a\in L_1^w(\M)$ is said to be an $(1,2)_{c}^w$-crude atom with respect to the filtration $(\M_k)_{k\geq1}$ and a weight $w$ if there exist $k\geq1$ and a factorization $a=yb$ such that
	\begin{enumerate}[{\rm (1)}]
		\item $y\in L_2^w(\M)$, $\E_k(y)=0$ and $\|y\|_{L_2^w(\M)}\leq 1$;
		\item $b\in L_2^w(\M_k)$ with $ \|b\|_{L_2^w(\M)}\leq 1$.	\end{enumerate}
Replacing the factorizations above by $a=by$, we have the notion of $(1,2)_{r}^w$-crude atoms.
Also, we define the associated weighted atomic spaces by substituting $ H_{1,\rm{at}}^{cw}(\M)$ with $ H_{1,\rm{crude}}^{cw}(\M)$ in Definition \ref{def-atomic-spaces}.
\end{definition}

The following lemmas can be deduced from \cite[Remark 3.7, Proposition 3.9]{CRX2023}. We include the details here for the sake of completeness.
\begin{lemma}\label{equiv1}
Every algebraic $h_{1w}^c$-atom $z$ can be rewritten as
$$z=\sum_{k}\lambda_k b_k,$$
where $b_k$'s are $(1,2)_c^w$-crude atoms and $\sum_{k}|\lambda_k|\leq 1$. Consequently, we have
$$
H_{1,\rm{aa}}^{cw}(\M)\subseteq H_{1,\rm{crude}}^{cw}(\M).
$$
\end{lemma}
\begin{proof}
Assume that $z$ is an algebraic $h_{1w}^c$-atom. Then it can be decomposed as
 $$z=\sum_{k\geq1}a_kb_k,$$
with $(a_k)_k$ and $(b_k)_k$ as in Definition \ref{def-alge-atom}. Let
$$
y_k=\frac{a_kb_k}{\|a_k\|_{L_2^w(\M)}\|b_k\|_{L_2^w(\M)}},
$$
and let
$$
\lambda_k=\|a_k\|_{L_2^w(\M)}\|b_k\|_{L_2^w(\M)}.
$$	
Then, for each $k$, it is evident that $y_k$ is an $(1,2)_c^w$-crude atom and $z$ can be rewritten as
$$
z=\sum_{k}\lambda_k y_k.
$$
By the H\"{o}lder inequality, we arrive at
\begin{align*}
\sum_k\lambda_k&=\sum_k\|a_k\|_{L_2^w(\M)}\|b_k\|_{L_2^w(\M)}\\
&\leq\Big(\sum_k\|a_k\|_{L_2^w(\M)}^2\Big)^{1/2}
\Big(\sum_k\|b_k\|_{L_2^w(\M)}^2\Big)^{1/2}\\
&\leq \Big\|\big(\sum_{k\geq 1}|b_k|^2\big)^{1/2}\Big\|_{L_2^w(\M)}\leq 1.
\end{align*}
This completes the proof.
\end{proof}

\begin{lemma}\label{equiv2}
Every $(1,2)_c^w$-crude atom $y$ can be decomposed as
$$y=\sum_{k}\lambda_k y_k,$$
where $y_k$'s are $(1,2)_c^w$-atoms and $\sum_{k}|\lambda_k|\leq 1$. Consequently, we have
$$
H_{1,\rm{crude}}^{cw}(\M)\subseteq H_{1,\rm{at}}^{cw}(\M).
$$
\end{lemma}

\begin{proof}
Let $y$ be a $(1,2)_c^w$-crude atom. Then there exists $n\geq 1$ such that $y=ab$ with $\|a\|_{L_2^w(\M)}\leq 1$ and $b\in L_2^w(\M_n)$, $\|b\|_{L_2^w(\M)}\leq 1$. We may assume that $b\geq 0$. For any fixed $l>1$, we consider the sequence of mutually disjoint projections $(e_k)_{k\in\mathbb{Z}}$ in $\M_n$ defined by
$$
e_k=\chi_{[l^k,l^{k+1})}(b).
$$
By the definition of spectral projection, we have $l^k e_k<e_kbe_k\leq l^{k+1}e_k $. For each $k\in\Z$, define $\lambda_k=\|abe_k\|_{L_2^w(\M)}\varphi(e_kw)^{1/2}$ and
$$
y_k=\varphi(e_kw)^{-1/2}\frac{abe_k}{\|abe_k\|_{L_2^w(\M)}}.
$$
Then, each $y_k$ becomes a $(1,2)_c^w$-atom and $y=\sum_k\lambda_ky_k$. Moreover, we estimate
$$
\|abe_k\|_{L_2^w(\M)}^2=\|a(e_kbe_k)\|_{L_2^w(\M)}^2\leq l^{2(k+1)}\|ae_k\|_{L_2^w(\M)}^2,
$$
and
$$
\|b\|_{L_2^w(\M)}^2=\Big\|\sum_{k\in\Z}e_kbe_k\Big\|_{L_2^w(\M)}^2\geq\Big\|\sum_{k\in\Z}l^ke_k\Big\|_{L_2^w(\M)}^2=\sum_{k\in\Z}l^{2k}\varphi(e_kw).
$$
Thus, combining the above estimates with the H\"{o}lder inequality, we have
\begin{align*}
\sum_k\lambda_k&
\leq \Big(\sum_k l^{2k}\varphi(e_kw)\Big)^{1/2}
\Big(\sum_k l^{-2k}\|abe_k\|_{L_2^w(\M)}^2\Big)^{1/2}\\
&\leq l\|b\|_{L_2^w(\M)}\Big(\sum_k\|ae_k\|_{L_2^w(\M)}^2\Big)^{1/2}\leq l\|a\|_{L_2^w(\M)}\leq l.
\end{align*}
We conclude this proof by letting $l\rightarrow 1$.
\end{proof}

\begin{proof}[Proof of \eqref{equiv-atomspaces}]
Indeed, a combination of Lemma \ref{equiv1} and Lemma \ref{equiv2} yields
$$
H_{1,\rm{aa}}^{cw}(\M)\subseteq H_{1,\rm{crude}}^{cw}(\M)\subseteq H_{1,\rm{at}}^{cw}(\M) .
$$
The converse inclusions can be verified easily with simple observations that $(1,2)_{c}^w$-atoms are $(1,2)_{c}^w$-crude atoms, and $(1,2)_{c}^w$-crude atoms are algebraic $h_{1w}^c$-atoms.

\end{proof}

\bibliographystyle{amsplain}
\bibliography{weighted_CZ}

\end{document}